\documentclass[onefignum,onetabnum]{siamonline220329}

\usepackage{bm}
\usepackage{booktabs}
\usepackage{bbold}
\usepackage{mathtools}
\usepackage{amssymb}
\usepackage{stmaryrd}
\usepackage{cleveref}
\usepackage{graphicx}
\usepackage{relsize}
\usepackage{subcaption}
\usepackage{algorithm}
\usepackage{algpseudocode}
\usepackage{nicefrac}
\usepackage{multirow}

\usepackage[draft, todonotes={textsize=tiny}]{changes}
\setuptodonotes{color=red, backgroundcolor=red!60!white,
  bordercolor=red, tickmarkheight=10pt}
  
\definechangesauthor[name=Robert, color=cyan]{RB}

\definechangesauthor[name=Florin, color=magenta]{FB}

\crefformat{equation}{(#2#1#3)}
\Crefformat{equation}{#2Equation#3~(#2#1#3)}
\crefrangeformat{equation}{(#3#1#4)--(#5#2#6)}
\Crefrangeformat{equation}{#3Equations#4~(#3#1#4)--(#5#2#6)}
\crefmultiformat{equation}{(#2#1#3)}{ and~(#2#1#3)}{, (#2#1#3)}{, and~(#2#1#3)}
\Crefmultiformat{equation}{#2Equations#3~(#2#1#3)}{ and~(#2#1#3)}{, (#2#1#3)}{, and~(#2#1#3)}
\crefrangemultiformat{equation}{(#3#1#4)--(#5#2#6)}{ and~(#3#1#4)--(#5#2#6)}{, (#3#1#4)--(#5#2#6)}{, and~(#3#1#4)--(#5#2#6)}
\Crefrangemultiformat{equation}{#3Equations#4~(#3#1#4)--(#5#2#6)}{ and~(#3#1#4)--(#5#2#6)}{, (#3#1#4)--(#5#2#6)}{, and~(#3#1#4)--(#5#2#6)}

\crefname{section}{Section}{Sections}
\crefname{appendix}{Appendix}{Appendices}
\crefname{theorem}{Theorem}{Theorems}
\crefname{proposition}{Proposition}{Propositions}
\crefname{corollary}{Corollary}{Corollaries}
\crefname{lemma}{Lemma}{Lemmata}
\crefname{remark}{Remark}{Remarks}
\crefname{algorithm}{Algorithm}{Algorithms}
\crefname{figure}{Figure}{Figures}
\crefname{table}{Table}{Tables}

\allowdisplaybreaks

\DeclareMathOperator{\diag}{diag}

\DeclareMathOperator*{\argmax}{argmax}

\DeclareMathOperator*{\supp}{supp}

\DeclareMathOperator{\id}{Id}

\DeclareMathOperator{\GW}{GW}

\DeclareMathOperator{\GWB}{GWB}
\DeclareMathOperator{\GWTB}{GWTB}
\DeclareMathOperator{\TB}{TB}
\newcommand{\tildeTB}{\widetilde\TB}

\DeclareMathOperator{\MGW}{MGW}

\DeclareMathOperator{\LGW}{LGW}

\DeclareMathOperator{\KL}{KL}

\DeclareMathOperator{\F}{\mathcal{F}}

\newcommand{\Fgwb}{\F_{\mathrm{GWB}_\rho}}

\newcommand{\Fgw}{F_{\mathrm{GW}}}
\newcommand{\Fmgw}{F_{\mathrm{MGW}_\rho}}

\newcommand{\Pio}{\Pi_{\mathrm{o}}}
\newcommand{\Beta}{\mathrm{B}}
\newcommand{\II}{\mathbb{I}}
\newcommand{\JJ}{\mathbb{J}}
\newcommand{\XX}{\mathbb{X}}
\newcommand{\YY}{\mathbb{Y}}
\newcommand{\ZZ}{\mathbb{Z}}

\newcommand{\rest}{\vert}

\newcommand{\T}{\mathfrak{T}}

\newcommand{\Exp}{\mathbb{E}\mathrm{xp}}
\newcommand{\Log}{\mathbb{L}\mathrm{og}}
\newcommand{\Melt}{\mathrm{M}}
\newcommand{\Melto}{\Melt^\mathrm{o}}
\newcommand{\Glue}{\Gamma}
\newcommand{\YYtimes}{\YY_{\!\!\times}}

\newcommand{\weakly}{\rightharpoonup}
\newcommand{\N}{\mathcal{N}}
\newcommand{\NN}{\ensuremath{\mathbb{N}}}
\newcommand{\M}{\mathcal{M}}

\newcommand{\R}{\ensuremath{\mathbb{R}}}

\newcommand{\Xf}{\mathfrak{X}}
\newcommand{\Yf}{\mathfrak{Y}}
\newcommand{\Zf}{\mathfrak{Z}}
\newcommand{\ff}{\mathfrak{f}}
\newcommand{\gf}{\mathfrak{g}}

\newcommand{\mf}{\mathfrak{m}}

\newcommand{\dx}{\,\mathrm{d}}
\newcommand{\tT}{\mathrm{T}}
\newcommand{\p}{\mathcal{P}}

\newcommand{\sym}{\mathrm{sym}}
\DeclareMathOperator{\borel}{\mathcal{B}}
\DeclareMathOperator{\pmeas}{\mathcal{P}}
\DeclareMathOperator{\gauges}{\mathfrak{G}}
\newcommand{\GM}{\mathfrak{G\!M}}
\newcommand{\para}{\mathfrak{P}}

\newcommand{\eps}{\varepsilon}

\newlength{\labwidth}

\numberwithin{equation}{section}

\headers{Tangential Fixpoint Iterations for
Gromov--Wasserstein Barycenters}
{F. Beier and R. Beinert}

\title{Tangential Fixpoint Iterations for
Gromov--Wasserstein Barycenters}

\author{Florian Beier%
  \thanks{Institute of Mathematics,
    Technische Universit\"at Berlin,
    Stra\ss{}e des 17. Juni 136,
    10623 Berlin, Germany 
    (\email{f.beier@tu-berlin.de},
    \email{robert.beinert@tu-berlin.de}).}
  \and Robert Beinert\footnotemark[1]}

\begin{document}

\maketitle

\begin{abstract}
The Gromov--Wasserstein (GW) transport problem
is a relaxation of classic optimal transport,
which seeks a transport between two measures
while preserving their internal geometry.
Due to meeting this theoretical underpinning,
it is a valuable tool 
for the analysis of objects that
do not possess a natural
embedding or should be studied 
independently of it.
Prime applications can thus be found in e.g.\
shape matching, classification and interpolation tasks.
To tackle the latter, 
one theoretically justified approach
is the employment of
multi-marginal GW transport
and GW barycenters,
which are Fréchet means with respect to
the GW distance.
However, because the computation of GW
itself already poses 
a quadratic and non-convex
optimization problem,
the determination of GW barycenters
is a hard task and 
algorithms for their computation are scarce.
In this paper, 
we revisit a known procedure
for the determination of Fréchet means
in Riemannian manifolds 
via tangential approximations
in the context of GW.
We provide a characterization of
barycenters in the GW tangent space,
which ultimately gives rise to 
a fixpoint iteration for 
approximating GW barycenters
using multi-marginal plans.
We propose a relaxation 
of this fixpoint iteration
and show that it monotonously
decreases the barycenter loss.
In certain cases our proposed method
naturally provides us with barycentric embeddings.
The resulting algorithm is capable 
of producing qualitative shape interpolations
between multiple 3d shapes 
with support sizes of over thousands of points
in reasonable time.
In addition, we verify our method on
shape classification 
and multi-graph matching tasks.
\end{abstract}

\section{Introduction}

Optimal transport (OT) focuses on
transporting one given measure to another
while minimizing some given cost.
The interest of OT is twofold.
Firstly, a solution to the problem, 
i.e.\ an optimal transport plan
gives a means of correspondence between the input measures.
Secondly, 
the value of the functional at a minimizer
also gauges their divergence 
with respect to the underlying cost function.
For a broad overview on the subject
we refer to \cite{PC19book}.
As its formulation is very general,
OT has applications in a wide range 
of topics such as
image matching \cite{wang2013linear,ZYHT07},
signal processing \cite{elvander2020multi}
and
particle dynamics \cite{KLNS20,junge2022entropic}.
Moreover,
due to relaxations and associated
efficient approximate solvers,
OT is now an established tool
in machine learning
\cite{ACB17,FZM15,KSKW15}.
The most prominent relaxation
is entropy-regularized OT
which can be solved by the
parallelizable Sinkhorn algorithm
\cite{C2013}.
Other efficient approximate solvers rely on
e.g.\
slicing strategies 
\cite{texturemix11,
quellmalz2023sliced,
quellmalz2024parallelly},
restriction to Gaussian mixture models 
\cite{delon2020wasserstein}
and
low-rank restrictions \cite{forrow2019statistical,scetbon2021low}.
A generalization of OT 
which concerns itself 
with transporting between more than two
input measures is multi-marginal OT \cite{CGEI2010}.
The resulting formulation can be used
to characterize OT barycenters
which are Fréchet means with respect to 
the OT divergence \cite{AC11barycenters}.
Independently of barycenters,
multi-marginal OT also finds applications 
in e.g.\
matching for teams \cite{CGEI2010},
particle tracking \cite{BLNS2021},
density functional theory \cite{Pass15}.
The metric version of OT, 
namely the Wasserstein distance,
has been a major focus 
in- and outside of the OT community.
Here measures are considered to be 
members of the same underlying metric space
and the associated cost function in
the transport problem is 
a power of a metric.
The Wasserstein space
is then obtained by
equipping the space of probability measures
with finite moments
with the Wasserstein distance.
The Wasserstein space
is a metric space and exhibits a
rich Riemannian structure 
which has been extensively studied in \cite{santambrogio2017euclidean,Ambrosio}
and gives rise to the very active study of
Wasserstein gradient flows \cite{hertrich2024wasserstein,
altekruger2023neural,
neumayer2024wasserstein,hagemann2024posterior,hertrich2023wasserstein,Ambrosio,
JKO1998,Ot01,Pav2014,
ArKoSaGr19,
EGNS2021}.

A shortcoming of the Wasserstein distance
is that it is heavily
dependent on the embedding 
of the given input measures.
One line of work,
namely Gromov--Wasserstein (GW) transport,
which has been sparked
by Sturm \cite{sturm2006geometry} focuses
on relaxations of OT to obtain an
embedding-free comparison 
and matching of 
so-called metric measure (mm-) spaces.
In addition to a measure,
an mm-space possesses
additional internal geometrical
structure in the form of a metric.
Sturms proposed transport problem 
is a twofold minimization over all 
measure-preserving isometric embeddings
of either input 
into a common ambient metric space 
and over associated 
Wasserstein transport plans.
The proposed transport problem 
has desirable properties such as
independence of 
isometric transformations 
of either inputs.
Although theoretically appealing,
solving the proposed problem in practice
is intractable.
With the same idea in mind, 
Mémoli \cite{memoli2011gromov}
proposed an alternative 
transport problem between mm-spaces
which is now widely known as the GW distance.
Although non-convex and quadratic, 
the problem is more accessible
from a numerical standpoint.
Because of this improved accessibility 
and since the GW distance 
meets the same theoretical
underpinnings as Sturms construction,
it is a valuable tool 
for matching and comparison tasks
where inputs do not possess a natural
embedding or should be studied 
independently of it.
The proposed distance and transport problem
thus has natural applications 
in e.g.\ shape matching \cite{memoli2011gromov},
graph analysis \cite{nguyen2022linearfused},
word alignment \cite{alvarez2018gromov}
and
particle dynamics \cite{beier2023gromov}.
Solving GW still poses 
a non-convex, 
quadratic optimization problem,
which motivated the OT community 
to propose several relaxations,
most of which are inspired from the OT case.
Examples are
entropy-regularized GW \cite{PCS2016,xu2019scalable},
low-rank constrained GW \cite{scetbon2022linear}
and sliced GW \cite{sliced_gw,BHS21}.
A fast exact solver for GW for inputs in
low-dimensional Euclidean space 
has been presented in \cite{ryner2023_globally}.
Recently,
multi-marginal GW (MGW) has been proposed 
which exhibits similar connections to GW barycenters, i.e. Fréchet means
with respect to the GW distance,
as in the Wasserstein setting
\cite{BBS2022multi}.
In \cite{sturm2020space}, 
Sturm showed that using the GW distance,
a metric space can be constructed which
exhibits a similar Riemannian structure
as the Wasserstein space.
This additional structure can be leveraged 
e.g.\ for achieving a computational speedup
when tasks require all 
pairwise distances of a large dataset.
This potential has been explored 
in the Wasserstein case
\cite{wang2013linear,
moosmuller2021linear}
as well as in the GW case 
\cite{beier2022linear}.
A fruitful iterative method for the
approximation of Fréchet means or barycenters
on Riemannian manifolds is as follows:
lift the inputs into the tangent space
at a reference point,
determine 
the associated tangent barycenter 
and update the reference point by
projecting back to the manifold
\cite{pennec}.
This Fréchet mean method 
has been successfully used
to obtain a simple fixpoint iteration
to approximate Wasserstein barycenters
and multi-marginal optimal transport
\cite{ABC16fixedpoint,lindheim2023simple}.
To the best of our knowledge,
the only work which follows a similar idea
in the GW case is \cite{chowdhury2020gromov}.
However,
in the reference,
the authors assume 
the given spaces to be finite
and then consider
a gradient descent 
of the Fréchet functional.
The nature of the provided proofs
build strongly on the assumed discreteness
and do not translate to the general case.
Moreover,
a characterization of Fréchet means
in the GW tangent space is not given.

In this work, 
we are motivated by approximating
GW barycenters via the previously mentioned
Fréchet mean method.
We study the tangent space and provide a characterization of tangential barycenters.
The latter gives rise to a theoretically 
justified fixpoint iteration 
for GW barycenter computation.
We propose a relaxation of it 
which has desirable properties.
Firstly, 
it can be easily implemented 
and only requires a solver for 
the GW transport problem.
Secondly, 
we show that it monotonously
decreases the GW barycenter functional.
Furthermore,
a single iteration of our proposed method
is often sufficient to obtain a barycenter,
e.g.\ when merely considering two inputs
or any number of Gaussian spaces 
endowed with the standard scalar product.
Existing state-of-the-art algorithms 
often have the undesirable property
that computed approximate barycenters 
do not come with an embedding.
Depending on the task,
this may require numerically expensive 
embedding techniques.
We discuss special cases, 
where our method naturally provides
us with an embedding of the barycenter.
We show an auxillary result 
which states that there exists a
GW barycenter 
between Gaussian spaces 
endowed with the standard scalar product 
which is again a Gaussian space.
We run numerical experiments,
showing that our method is capable 
of producing qualitative shape interpolations
between multiple 3d shapes 
with support sizes of over 5000 points.
To the best of our knowledge,
no existing GW-based method 
is able to achieve this.
Our experiments indicate that, 
in practice, 
running our method once
produces 
the entire interpolation
between the given inputs.
In addition,
we show that our method can be used to
classify 3d shapes and to obtain
multi-graph matchings.

\paragraph{Main Contributions}
\begin{itemize}
    \item
    Justification
    of the Fréchet mean method for the GW case
    in \cref{thm:tan_min}.
    More precisely,
    we show that
    the required barycenters
    in the GW tangent space always exist
    and can be characterized 
    by multi-marginal plans.
    \item 
    Relaxation of the Fréchet mean method
    in the form of a simple-to-implement
    fixpoint iteration.
    We show that this iteration decreases
    the barycenter loss monotonously,
    see \cref{thm:G_monotone}.
    Furthermore, \cref{thm:G_fixpoint} shows that every subsequence of this fixpoint iteration contains a converging subsequence 
    whose limit is an actual fixpoint.
    Notice that every GW barycenter is a fixpoint
    of the relaxed Fréchet mean procedure.
    \item 
    For Gaussian distributions 
    endowed with the Euclidean scalar product,
    we analytically 
    characterize a GW barycenter as
    a Gaussian distribution 
    with the same Euclidean scalar product,
    see \cref{thm:gaussian}.
\end{itemize}

Our paper is organized as follows.
\cref{sec:2}
provides the reader with 
the fundamental definitions
and preliminary results
related to GW transport.
In \cref{sec:3},
we define the multi-marginal formulation
of the GW problem.
Following \cite{sturm2020space},
\cref{sec:4}
discusses the Riemmannian structure 
and gives the definition of the tangent space.
In \cref{sec:5}
we introduce the GW barycenter problem
and show its one-to-one correspondence
to multi-marginal GW transport plans.
We proceed with the definition of the
tangential GW barycenter problem
and characterize its solutions.
In \cref{sec:6},
we propose a relaxed fixpoint iteration
to approximate GW barycenters 
by a sequence
of projected tangential barycenters.
We show that the sequence 
monotonously decreases the GW barycenter loss.
\cref{sec:7} shows how the fixpoint iteration 
can be algorithmically implemented
and discusses related practicalities.
\cref{sec:8} provides three numerical experiments 
indicating the potential of our proposed method.

\section{Gauged Measure Spaces and Gromov--Wasserstein}\label{sec:2}
For any Polish space $X$,
we denote the related Borel $\sigma$-algebra by $\borel (X)$,
the set of signed Borel measures by $\M(X)$,
the set of positive Borel measures by $\M^+(X)$,
and the set of Borel probability measures by $\pmeas(X)$.
If $\Phi \colon X \to Y$ is Borel measurable 
between the Polish spaces $X$ and $Y$,
the \emph{push forward} 
of $\xi \in \pmeas(X)$ by $\Phi$ 
is defined via
\begin{equation*}
    (\Phi_\# \xi) (A)
    \coloneqq 
    (\xi \circ \Phi^{-1})(A),
    \quad A \in \mathcal B(Y).
\end{equation*}
Depending on $\xi \in \pmeas(X)$,
any symmetric and square-integrable function 
$g \colon X \times X \to \R$
with respect to the product measure $\xi \otimes \xi$
is called a \emph{gauge}.
The \emph{set of all gauges}
on $X$ with respect to $\xi$
is denoted by 
$\gauges(X,\xi) \coloneqq L_\sym^2(X \times X, \xi \otimes \xi)$.
For every map $\Psi \colon Y \to X$, 
the \emph{pull back} 
of $g$ through $\Psi$ 
is defined as
\begin{equation*}
    (\Psi^\# g) (y,y') 
    \coloneqq
    g(\Psi(y), \Psi(y')),
    \quad y, y' \in Y.
\end{equation*}
Any triple $\XX \coloneqq (X,g,\xi)$ 
consisting of
\begin{itemize}
    \item a Polish space $X$,
    \item a measure $\xi \in \pmeas(X)$,
    \item a gauge $g \in \gauges(X,\xi)$
\end{itemize} 
is called a \emph{gauged measure space (gm-space)}.
Figuratively,
the measure describes
how the space is weighted 
whereas
the gauge provides geometrical information.
An important case of gm-spaces are obtained by
choosing a metric as gauge.
These spaces are also called metric measure spaces (mm-spaces).
Popular other choices are 
(powers of) metrics,
inner products, 
and adjacency matrices
in the realm of graph matching.

The Gromov--Wasserstein (GW) distance 
between two gm-spaces 
$\XX \coloneqq (X, g, \xi)$
and
$\YY \coloneqq (Y, h, \upsilon)$
is an optimal-transport-based pseudometric.
Henceforth,
we denote the \emph{set of transport plans} 
between $\XX$ and $\YY$ by
$\Pi(\XX,\YY)
\coloneqq 
\{ \pi \in \pmeas(X \times Y) :
(P_X)_\# \pi = \xi,
(P_Y)_\# \pi = \upsilon \}$,
where $P_\bullet$ signifies the projection 
to the indicated component.
The \emph{GW-2} or just \emph{GW distance} is given as
\begin{equation}
    \label{eq:GW}
    \GW_2(\XX,\YY) 
    \coloneqq
    \inf_{\pi \in \Pi(\XX,\YY)}
    \Fgw^{\XX,\YY} (\pi)
\end{equation}
with the GW functional
\begin{align}
    \Fgw^{\XX,\YY} (\pi)
    &\coloneqq
    \|g(\cdot_1,\cdot_3) - h(\cdot_2,\cdot_4)\|
    _{L^2_\sym(\pi(\cdot_1,\cdot_2) \otimes \pi(\cdot_3,\cdot_4))}
    \notag\\
    \label{eq:gw-fun}
    &\coloneqq
    \Bigl(
    \iint_{(X \times Y)^2} 
    \rvert
    g(x,x') - h(y,y')
    \rvert^2 
    \dx \pi(x,y) \dx \pi(x',y') 
    \Bigr)^{\frac{1}{2}}.
\end{align}
The infimum in \cref{eq:GW} is always attained \cite[Prop~5.4]{sturm2020space},
and
we denote the \emph{set of (GW-) optimal transport plans} 
by $\Pio(\XX,\YY)$.
Two gm-spaces $\XX$ and $\YY$ are called \emph{homomorphic}
($\XX \simeq \YY$)
if $\GW(\XX,\YY) = 0$.
In this case,
there exist a third gm-space 
$\ZZ \coloneqq (Z,f,\zeta)$
as well as
Borel-measurable maps 
$\Phi \colon Z \to X$
and
$\Psi \colon Z \to Y$
such that
$\xi = \Phi_\# \zeta$,
$\upsilon = \Psi_\# \zeta$,
and
$f = \Phi^\# g = \Psi^\# h$,
see \cite[Prop~5.6]{sturm2020space}.
In particular,
$\XX \simeq \YY$
if there exists a map $\Phi \colon X \to Y$
such that
$\upsilon = \Phi_\# \xi$ and $g = \Phi^\# h$.
The homomorphic equivalence class of $\XX$ is denoted by
$\Xf \coloneqq \llbracket \XX \rrbracket$,
and the \emph{space of homomorphic gm-spaces} is expressed as $\GM$.
The GW distance defines a metric on $\GM$.
Moreover,
$(\GM,\GW)$ is a complete, geodesic metric space,
see \cite[Thm~5.8]{sturm2020space}.

Every gm-space $\XX \coloneqq (X, g, \xi)$ is a Lebesgue--Rokhlin space,
which means that $\xi \in \pmeas(X)$ can be characterized
via the Lebesgue measure $\lambda$ on the unit interval $[0,1]$,
cf.\ \cite[Lem~1.15]{sturm2020space}.
More precisely,
there exists a Borel-measurable map $\Phi\colon [0,1] \to X$
so that $\Phi_\# \lambda = \xi$.
The map $\Phi$ may be chosen with respect to the atomic decomposition
with finite or infinite many atoms:
\begin{equation*}
    \label{eq:atom-deco}
    \xi = \sum_{n=1}^\infty \xi_n \delta_{x_n} + \tilde{\xi},
\end{equation*}
where $\xi_n \in [0,1]$ is the weight of the Dirac measure at $x_n \in X$,
and $\tilde{\xi} \in \M^+(X)$ the remaining diffuse part. 
On the basis of the partition 
\begin{equation}
    \label{eq:para}
    I_n \coloneqq [M_{n-1}, M_n)
    \quad\text{and}\quad
    I_\infty \coloneqq [M_\infty, 1]
    \quad\text{with}\quad
    M_n \coloneqq \sum_{m=1}^n \xi_m, 
    \quad n \in \NN \cup \{\infty\},
\end{equation}
the \emph{parametrization} $\Phi$ of $\XX$ 
can be chosen such that
$\Phi(I_n) \equiv x_n$
and that
the restriction $\Phi |_{I_\infty} \colon I_\infty \to \supp (\tilde\xi)$
is bijective with Borel-measurable inverse.
The \emph{set of parametrizations} of $\XX$ is expressed as $\para(\XX)$
and is independent of the actual gauge.
For any $\Phi \in \para(\XX)$,
$\XX$ is obviously homomorphic to
$\II \coloneqq ([0,1],\bar{g},\lambda)$ with $\bar{g} \coloneqq \Phi^\# g$,
which essentially allows us to restrict ourselves to gm-spaces over $[0,1]$.
The next result shows that the parametrization can be carried over 
to (optimal) transport plans between arbitrary gm-spaces.

\begin{lemma}
    \label{lem:bi-par}
    Let the gm-spaces 
    $\XX \coloneqq (X,g,\xi)$ and $\YY \coloneqq (Y,h,\upsilon)$ 
    be homomorphic to
    $\II = ([0,1],\bar{g},\lambda)$ and $\JJ = ([0,1],\bar{h},\lambda)$ 
    via the parametrizations 
    $\Phi \in \para(\XX)$ and $\Psi \in \para (\YY)$ respectively.
    Then
    \begin{equation*}
    \Pi(\XX,\YY) = (\Phi \times \Psi)_\# \Pi(\II,\JJ)
    \quad \text{and} \quad
    \Pio(\XX,\YY) = (\Phi \times \Psi)_\# \Pio(\II,\JJ).
    \end{equation*}
\end{lemma}

\begin{proof}
    On the basis of the following atomic decompositions
    $\xi = \sum_{n=1}^\infty \xi_n \delta_{x_n} + \tilde \xi$
    and 
    $\upsilon = \sum_{m=1}^\infty \upsilon_m \delta_{y_m} + \tilde \upsilon$,
    any $\pi \in \Pi(\XX,\YY)$ admits the form
    \begin{equation}
        \label{eq:atom-plan}
        \pi
        =
        \sum_{n,m = 1}^\infty \pi_{n,m} \delta_{(x_n,y_m)} 
        + \tilde \pi
        \quad\text{with}\quad
        \tilde \pi
        =
        \sum_{n=1}^\infty \tilde \pi_{n,\infty}
        + \sum_{m=1}^\infty \tilde \pi_{\infty,m}
        + \tilde \pi_{\infty,\infty},
    \end{equation}
    where the weight $\pi_{n,m}$ describes the partial transport 
    from $\delta_{x_n}$ to $\delta_{y_m}$,
    the measure $\tilde \pi_{n,\infty}$ the partial transport
    from $\delta_{x_n}$ to $\tilde \upsilon$,
    the measure $\tilde \pi_{\infty,m}$ the partial transport
    from $\tilde \xi$ to $\delta_{y_m}$,
    and the measure $\tilde\pi_{\infty,\infty}$ the partial transport 
    from $\tilde \xi$ to $\tilde\upsilon$.
    For the non-atomic parts $\tilde\pi_{\bullet,\bullet}$,
    the total masses are expressed by 
    $\pi_{\bullet,\bullet} \coloneqq \tilde\pi_{\bullet,\bullet}(X \times Y)$.    
    Without loss of generality,
    let $\Phi \in \para(\XX)$ and $\Psi\in\para(\YY)$ be chosen
    with respect to the atomic decomposition of $\xi$ and $\upsilon$
    as described in \cref{eq:para}.
    Recall that
    $\Phi$ and $\Psi$ are bijective
    with Borel-measurable inverse
    on $\supp(\tilde\xi)$ and $\supp(\tilde\upsilon)$.
    To distinguish the spaces $[0,1]$ of $\II$ and $\JJ$,
    we write $\II = (I, \bar g, \lambda)$
    with partition $(I_n)_{n\in\NN} \cup (I_\infty)$
    and $\JJ = (J, \bar h, \lambda)$
    with $(J_m)_{m\in\NN} \cup (J_\infty)$.
    For any $\pi \in \Pi(\XX,\YY)$,
    we now define $\bar{\pi} \in \pmeas(I \times J)$ via
    \begin{align*}
        \bar{\pi}\rest_{I_n \times J_m} 
        &\coloneqq 
        \tfrac{\pi_{n,m}}{\xi_n \upsilon_m} (\lambda\rest_{I_n} \otimes \lambda\rest_{J_m}),
        &
        \bar{\pi}\rest_{I_n \times J_\infty} 
        &\coloneqq 
        \tfrac{1}{\xi_n} 
        (\lambda\rest_{I_n} \otimes \Psi^{-1}_\# (P_Y)_\# \tilde\pi_{n,\infty}),
        \\
        \bar{\pi}\rest_{I_\infty \times J_m} 
        &\coloneqq 
        \tfrac{1}{\upsilon_m}
        (\Phi^{-1}_\# (P_X)_\# \tilde\pi_{\infty,m} \otimes \lambda\rest_{J_m}),
        &
        \bar{\pi}\rest_{I_\infty \times J_\infty} 
        &\coloneqq 
        (\Phi^{-1} \times \Psi^{-1})_\# \tilde \pi_{\infty,\infty}.
    \end{align*}
    The marginals of $\bar\pi$ are given by
    \begin{align*}
        (P_I)_\# \bar{\pi} \rest_{I_n} 
        &=
        \tfrac{1}{\xi_n}
        \Bigl( \sum_{n=1}^\infty \pi_{n,m} + \pi_{n,\infty} \Bigr)
        \lambda\rest_{I_n}
        =
        \lambda\rest_{I_n},
        \quad n\in\NN,
        \\
        (P_I)_\# \bar{\pi} \rest_{I_\infty}
        &=
        \sum_{m=1}^\infty \Phi^{-1}_\# (P_X)_\# \tilde\pi_{\infty,m}
        + \Phi^{-1}_\# (P_X)_\# \tilde\pi_{\infty,\infty}
        =
        \Phi^{-1}_\# \tilde\xi
        =
        \lambda\rest_{I_\infty},
    \end{align*}
    and analogously $(P_J)_\# \bar{\pi} = \lambda$,
    which shows $\bar{\pi} \in \Pi(\II,\JJ)$.
    Computing $(\Phi \times \Psi)_\# \bar\pi_{I_\bullet \times J_\bullet}$,
    we obtain all terms in \cref{eq:atom-plan}
    yielding
    $\Pi(\XX,\YY) \subset (\Phi \times \Psi)_\# \Pi(\II,\JJ)$.
    To show the opposite inclusion,
    we calculate the marginals of
    $(\Phi \times \Psi)_\# \bar{\pi}$
    for any $\bar{\pi} \in \Pi(\II,\JJ)$,
    which are
    \begin{equation*}
        (P_X)_\# (\Phi \times \Psi)_\# \bar \pi
        =
        \Phi_\# (P_I)_\# \bar \pi
        =
        \Phi_\# \lambda
        =
        \xi
    \end{equation*}
    and similarly $(P_{Y})_\# \pi = \upsilon$.
    Hence $(\Phi \times \Psi)_\# \Pi(\II,\JJ) = \Pi(\XX,\YY)$.
    Finally,
    for each pair
    $\pi \in \Pi(\XX,\YY)$ 
    and
    $\bar{\pi} \in \Pi(\II,\JJ)$ 
    with $\pi = (\Phi,\Psi)_\# \bar{\pi}$, 
    we have
    \begin{equation*}
        \Fgw^{\XX,\YY}(\pi)
        =
        \lVert g(\cdot_1, \cdot_3) - h(\cdot_2, \cdot_4) \rVert_{L^2_\sym(\pi \otimes \pi)}
        =
        \lVert \bar g(\cdot_1, \cdot_3) - \bar h(\cdot_2, \cdot_4) \rVert
        _{L^2_\sym(\bar \pi \otimes \bar \pi)}
        =
        \Fgw^{\II,\JJ}(\bar\pi).
    \end{equation*}
    Since the GW distance is independent of the representative,
    meaning $\GW(\XX,\YY) = \GW(\II,\JJ)$,
    we thus have 
    $\Pio(\XX,\YY) = (\Phi \times \Psi)_\# \Pio(\II,\JJ)$.
\end{proof}

The GW distance between $\Xf\in\GM$ and $\Yf \in \GM$ is,
by construction,
independent of the current representative,
i.e.\ $\GW(\Xf,\Yf) \coloneqq \GW(\XX,\YY)$
for any $\XX \in \Xf$ and $\YY \in \Yf$.
In order to transfer an (optimal) plan $\pi \in \Pi(\XX,\YY)$
between different representatives,
we use gluings and meltings as an alternative to the parametrizations above.
For the plans $\pi_1 \in \Pi(\XX, \YY)$ and $\pi_2 \in \Pi(\XX,\ZZ)$
between the gm-spaces 
$\XX \coloneqq (X,g,\xi)$,
$\YY \coloneqq(Y, h, \upsilon)$,
and $\ZZ \coloneqq (Z, f, \zeta)$,
the \emph{set of gluings} along $\XX$ is defined by
\begin{equation*}
    \Glue_\XX(\pi_1,\pi_2) 
    \coloneqq
    \Glue_\XX^{\YY,\ZZ}(\pi_1,\pi_2) 
    \coloneqq
    \bigl\{
    \gamma \in \p(X \times Y \times Z)
    : 
    (P_{X \times Y})_\# \gamma = \pi_1,
    (P_{X \times Z})_\# \gamma = \pi_2
    \bigr\}.
\end{equation*}
Due to Dudley's lemma \cite[Lem~8.4]{ABS21},
the set of gluings is in particular non-empty.
By construction,
the 2-marginal $(P_{Y \times Z})_\# \gamma$ 
of any gluing $\gamma \in \Gamma_\XX(\pi_1,\pi_2)$
lies in $\Pi(\YY,\ZZ)$.
This plan is also called a \emph{melting} of $\pi_1$ and $\pi_2$,
cf.\ \cite{sturm2020space}.
More generally,
the \emph{set of meltings} along $\XX$ is defined by
\begin{equation*}
    \Melt_\XX(\pi_1,\pi_2) 
    \coloneqq
    \Melt_\XX^{\YY,\ZZ}(\pi_1,\pi_2) 
    \coloneqq
    \bigl\{
    (P_{Y \times Z})_\# \gamma 
    : 
    \gamma \in \Glue_{\XX}(\pi_1,\pi_2)
    \bigr\}.
\end{equation*}
The sets of gluings and meltings can be extended
to arbitrary 
$\Omega_1 \subset \Pi(\XX,\YY)$
and $\Omega_2 \subset \Pi(\XX,\ZZ)$
by setting
\begin{equation*}
    \Glue_\XX(\Omega_1,\Omega_2) 
    \coloneqq 
    \bigcup_{\substack{\pi_1 \in \Omega_1 \\ \pi_2 \in \Omega_2}} 
    \Glue_\XX(\pi_1,\pi_2) 
    \quad\text{and}\quad 
    \Melt_\XX(\Omega_1,\Omega_2)
    \coloneqq
    \bigcup_{\substack{\pi_1 \in \Omega_1 \\ \pi_2 \in \Omega_2}}
    \Melt_\XX(\pi_1,\pi_2).
\end{equation*}
In the following,
we will mainly use gluings and meltings 
to transfer (optimal) plans between different representatives.

\begin{lemma}
    \label{lem:melting_remains_optimal}
    Let $\XX$, $\tilde \XX$, and $\YY$ be gm-spaces
    with $\XX \simeq \tilde \XX$.
    For any $\sigma \in \Pio(\XX,\tilde \XX)$,
    it holds
    \begin{equation*}
        \Pi(\tilde \XX,\YY) = \Melt_\XX(\sigma,\Pi(\XX,\YY))
        \quad\text{and}\quad 
        \Pio(\tilde \XX,\YY) = \Melt_\XX(\sigma,\Pio(\XX,\YY)).
    \end{equation*}
\end{lemma}

\begin{proof}
    Let $\XX \coloneqq (X, g, \xi)$,
    $\tilde \XX \coloneqq (\tilde X, \tilde g, \tilde\xi)$,
    and $\YY \coloneqq (Y, h, \upsilon)$.
    As elaborated above, 
    we obtain 
    $\Melt_\XX(\sigma,\pi) \subset \Pi(\tilde\XX,\YY)$ 
    for all $\pi \in \Pi(\XX,\YY)$,
    which ensures 
    $\Melt_\XX(\sigma,\Pi(\XX,\YY)) 
    \subset \Pi(\tilde\XX,\YY)$.
    The other way round, 
    let $\tilde \pi \in \Pi(\tilde \XX,\YY)$
    and $\gamma \in \Glue_{\tilde \XX} (\sigma^\tT, \tilde\pi)$,
    where $\sigma^\tT \in \Pio(\tilde \XX, \XX)$ denotes the reversion of 
    $\sigma \in \Pio(\XX,\tilde\XX)$.
    By construction,
    $\pi \coloneqq (P_{X \times Y})_\# \gamma \in \Pi(\XX,\YY)$
    and $(P_{X \times \tilde{X}})_\# \gamma = \sigma$.
    Hence,
    $\gamma \in \Glue_\XX(\sigma,\pi)$
    and finally
    $\tilde{\pi} = (P_{\tilde{X}\times Y})_\# \gamma \in \Melt_{\XX}(\sigma,\pi)$.
    Thus, we obtain the first identity.
    For the second identity,
    let $\pi \in \Pio(\XX,\YY)$,
    we consider $\tilde\pi \in \Melt_\XX(\sigma,\pi)$
    with underlying gluing $\gamma \in \Glue_\XX(\sigma,\pi)$.
    Since $g(\cdot_1,\cdot_3) = \tilde g(\cdot_2,\cdot_4)$
    a.e.\ with respect to 
    $\sigma(\cdot_1,\cdot_2) \otimes \sigma(\cdot_3,\cdot_4)$ 
    and hence 
    $g(P_X(\cdot_1), P_X(\cdot_3)) 
    = \tilde g(P_{\tilde X}(\cdot_2), P_{\tilde X}(\cdot_4))$ 
    a.e.\ with respect to 
    $\gamma(\cdot_1,\cdot_2) \otimes \gamma(\cdot_3,\cdot_4)$,
    we have
    \begin{align*}
        \GW(\XX,\YY)
        &= \Fgw^{\XX,\YY}(\pi)
        = \iint_{(X \times \tilde X \times Y)^2}
        \lvert g(x,x') - h(y,y') \rvert^2 
        \dx \gamma(x, \tilde x, y) \dx \gamma(x', \tilde x', y')
        \\
        &= \iint_{(X\times \tilde X \times Y)^2}
        \lvert \tilde g(\tilde x, \tilde x') - h(y,y') \rvert^2 
        \dx \gamma(x, \tilde x, y) \dx \gamma(x', \tilde x', y')
        = \Fgw^{\tilde\XX,\YY}(\tilde\pi)
        \ge \GW(\tilde \XX,\YY).
    \end{align*}
    Due to $\XX \simeq \tilde \XX$,
    the last inequality has to be tight
    showing $\tilde \pi \in \Pio(\tilde\XX,\YY)$.
    The opposite inclusion follows in an analogous way 
    as in the first part.
\end{proof}

If we melt a self-coupling $\sigma \in \Pio(\XX,\XX)$
with an optimal plan $\pi \in \Pio(\XX,\YY)$,
the resulting plans again lies in $\Pio(\XX,\YY)$. 
Every pair of plans 
that can be constructed in this manner
are called equivalent,
i.e.\
we say that $\pi_1,\pi_2 \in \Pio(\XX,\YY)$ 
are \emph{equivalent} with respect to $\XX$, 
written $\pi_1 \simeq_\XX \pi_2$, 
if there exists $\sigma \in \Pio(\XX,\XX)$ 
so that $\pi_1 \in \Melt_\XX(\sigma,\pi_2)$
and thus $\pi_2 \in \Melt_\XX(\sigma^\tT,\pi_1)$.

\begin{proposition}\label{prop:sim_eq_continuous}
    The relation $\simeq_\XX$ is weakly continuous, 
    i.e.\ if $\pi_2 \in \Pio(\XX,\YY)$ 
    and $\pi_{1,n} \simeq_\XX \pi_2$ for $n \in \NN$
    with $\pi_{1,n} \weakly \pi_1$,
    then $\pi_1 \simeq_\XX \pi_2$.
\end{proposition}

\begin{proof}
    Without loss of generality, 
    we consider the two copies
    $\XX_1 \coloneqq (I_1,\bar{g},\lambda)$
    and $\XX_2 \coloneqq (I_2,\bar{g},\lambda)$
    of $\XX$, where $I_1 = I_2 = [0,1]$.
    In the same way,
    let $\YY = (J,\bar{h},\lambda)$ with $J = [0,1]$.
    Since $\pi_{1,n} \simeq_{\XX} \pi_2$,
    we have
    $\pi_{1,n} \in \Pio(\XX_1,\YY)$.
    Furthermore, due to \cite[Lem.~5.5]{sturm2020space},
    it holds
    \begin{equation*}
    \Fgw^{\XX_1,\YY}(\pi_1) 
    = \lim_{n \to \infty} \Fgw^{\XX_1,\YY}(\pi_{1,n})
    = \GW(\XX_1,\YY),
    \end{equation*}
    so that $\pi_1 \in \Pio(\XX_1,\YY)$.
    Let $\sigma_n \in \Pio(\XX_1,\XX_2)$
    be so that 
    $\pi_{1,n} \in \Melt_{\XX_2}(\sigma_n,\pi_2)$.
    We denote the 
    corresponding gluing by
    $\gamma_n \in \Gamma_{\XX_2}(\sigma_n,\pi_2)$,
    i.e.\ $(P_{I_1 \times J})_\# \gamma_n = \pi_{1,n}$.
    Since $(\gamma_n)_{n \in \NN} \subset \Pi(\XX_1,\XX_2,\YY)$
    and the latter is weakly compact, we may extract a converging subsequence, i.e.\
    $\gamma_{n_\ell} \weakly \gamma$ as $\ell \to \infty$.
    The marginal projections are weakly continuous,
    hence we obtain
    $(P_{I_1 \times J})_\# \gamma 
    = \lim_{\ell \to \infty} \pi_{1,n_\ell} = \pi_1$
    as well as
    $\sigma 
    \coloneqq (P_{I_1 \times I_2})_\# \gamma
    = \lim_{\ell \to \infty} \sigma_{n_\ell} \in \Pio(\XX_1,\XX_2)$.
    The optimality of $\sigma$ can be shown
    analogously to the optimality of $\pi_1$ above.
    Finally, this gives
    $\gamma \in \Gamma_\YY(\sigma,\pi_1)$
    with
    $(P_{I_2 \times J})_\# \gamma = \pi_2$
    as desired.
\end{proof}

\section{Multi-Marginal GW}\label{sec:3}

We focus on a multi-marginal formulation of \cref{eq:GW},
which has recently been proposed in \cite{BBS2022multi}.
To introduce the simultaneous transport between
the gm-spaces $\XX_i \coloneqq (X_i, g_i, \xi_i)$ with $i=1,\dots,N$,
we denote the Cartesian product over the domains as 
$X_\times \coloneqq \bigtimes_{i=1}^N X_i$
with elements $x_\times \coloneqq (x_1,\dots,x_N)$, $x_i \in X_i$,
and the \emph{set of multi-marginal transport plans} as
$\Pi(\XX_1,\dots,\XX_N)
\coloneqq 
\{\pi \in \pmeas(X_\times) : (P_{X_i})_\# \pi = \xi_i\}$.
The \emph{$(N-1)$-dimensional probability simplex} is characterized by 
\[
\Delta_{N-1} 
\coloneqq 
\{(\rho_1,\dotsc,\rho_N) \in [0,1] : \sum_{i=1}^N \rho_i = 1\}.
\]
For $\rho \in \Delta_{N-1}$, 
we define the 
\emph{multi-marginal GW (MGW) problem} by
\begin{equation}
    \label{eq:mgw2}
    \MGW_\rho(\XX_1,\dots,\XX_N)
    \coloneqq
    \inf_{\pi \in \Pi(\XX_1,\dots,\XX_N)}
    \Fmgw^{\XX_1,\dots,\XX_N} (\pi)
\end{equation}
with the MGW functional
\begin{equation}
    \label{eq:mgw-fun}
    \Fmgw^{\XX_1,\dots,\XX_N} (\pi)
    \coloneqq
    \iint_{X_\times^2} 
    \frac{1}{2}
    \sum_{i,j=1}^N \rho_i\rho_j \; 
    \lvert g_i(x_i,x_i') - g_j(x_j,x'_j) \rvert^2
    \dx \pi(x_\times) \dx \pi(x'_\times).
\end{equation}
The reason why we restrict ourselves to
weights of the multiplicative form $\rho_i \rho_j$ 
in the respective summand will become evident 
when we introduce 
the free-support GW barycenter problem later on.
Depending on the situation more general weights
of the form $w_{i,j}$ may be considered.
The \emph{set of optimal plans} 
with respect to \cref{eq:mgw2}
is denoted by
$\Pio^\rho(\XX_1,\dots,\XX_N)$,
although we usually omit the superscript $\rho$ 
if it is clear from context.
The identities in \cref{lem:bi-par} 
generalize to the multi-marginal setting.

\begin{lemma}
    \label{lem:mul-par}
    Let the gm-spaces $\XX_i \coloneqq (X_i, g_i, \xi_i)$
    be homomorphic to $\II_i \coloneqq ([0,1],\bar g_i, \lambda)$
    via the parametrization $\Phi_i \in \para(\XX_i)$
    with $n=1,\dots,N$.
    Then
    \begin{align*}
        \Pi(\XX_1, \dots, \XX_N) 
        &=
        (\Phi_1 \times \cdots \times \Phi_N)_\# 
        \Pi(\II_1,\dots,\II_N)
        \shortintertext{and}
        \Pio(\XX_1, \dots, \XX_N) 
        &=
        (\Phi_1 \times \cdots \times \Phi_N)_\# 
        \Pio(\II_1,\dots,\II_N)
    \end{align*}
\end{lemma}

\begin{proof}
    The first identity can be established 
    using a similar construction as in the proof of \cref{lem:bi-par}
    by considering the transports between 
    the different atomic and diffuse parts.
    The second identity follows from
    $\Fmgw^{\XX_1,\dots,\XX_N}(\pi) = \Fmgw^{\II_1,\dots,\II_N}(\bar \pi)$
    for every $\pi = (\Phi_1 \times \cdots \times \Phi_N)_\# \bar \pi$.
\end{proof}

An immediate consequence of the previous lemma is that
the MGW problem is independent of the actual representation
of the homomorphic classes---%
a property well known for the GW distance. 

\begin{proposition}
    \label{prop:inv-hom}
    Let $\XX_i \simeq \tilde\XX_i$ for $i=1,\dots,N$.
    Then 
    \begin{equation*}
        \MGW_\rho(\XX_1,\dots,\XX_N) 
        = \MGW_\rho(\tilde \XX_1,\dots, \tilde \XX_N).
    \end{equation*}
    In particular, 
    $\MGW_\rho(\Xf_1,\dots,\Xf_N) \coloneqq \MGW_\rho(\XX_1,\dots,\XX_N)$ with $\Xf_i = \llbracket \XX_i \rrbracket \in \GM$ is well-defined.
\end{proposition}

\begin{proof}
    The homomorphic equivalence of
    $\XX_i = (X_i, g_i, \xi_i)$
    and $\tilde \XX_i = (\tilde X_i, \tilde g_i, \tilde \xi_i)$
    ensures the existence of a third gm-space 
    $\hat\XX_i = (\hat X_i, \hat g_i, \hat \xi_i)$
    and Borel-measurable maps
    $\Psi_i \colon \hat X_i \to X_i$ and 
    $\tilde\Psi_i \colon \hat X_i \to \tilde X_i$
    such that 
    $\xi_i = (\Psi_i)_\# \hat \xi_i$,
    $\tilde \xi_i = (\tilde \Psi_i)_\# \hat \xi_i$,
    and
    $\hat g_i = \Psi_i^\# g_i = \tilde\Psi_i^\# \tilde g_i$.
    Any $\Theta_i \in \para(\hat \XX_i)$ induces
    the parametrizations 
    $\Phi_i \coloneqq \Psi_i \circ \Theta_i \in \para(\XX_i)$
    and
    $\tilde\Phi_i \coloneqq \tilde\Psi_i \circ \Theta_i \in \para(\tilde\XX_i)$
    such that
    $\XX_i$ and $\tilde\XX_i$ are homomorphic to 
    $\II_i \coloneqq ([0,1], \bar g_i, \lambda)$
    with $\bar g_i = \Phi_i^\# g_i = \tilde\Phi_i^\# \tilde g_i$.
    The assertion now follows from 
    \begin{equation*}
            \Fmgw^{\XX_1,\dots,\XX_N}(\pi) 
            =\Fmgw^{\II_1,\dots,\II_N}(\bar\pi)
            =\Fmgw^{\tilde \XX_1,\dots,\tilde \XX_N}(\tilde \pi)
    \end{equation*}
    for
    $\pi = (\Phi_1 \times \cdots \times \Phi_N)_\# \bar\pi$
    and
    $\tilde \pi 
    = (\tilde \Phi_1 \times \cdots \times \tilde \Phi_N)_\# \bar\pi$
    with $\bar\pi \in \Pi(\II_1,\dots,\II_N)$,
    and since \cref{lem:mul-par} further implies the equality 
    of the infimum over all three MGW functionals.
\end{proof}

Through the invariance of the MGW problem with respect the homomorphic classes,
the infimum over the MGW functional \cref{eq:mgw-fun} is attained.

\begin{theorem}\label{thm:min-exists}
    For the gm-spaces $\XX_1, \dots, \XX_N$,
    there exists a multi-marginal plan minimizing \cref{eq:mgw-fun},
    i.e.\ $\Pio(\XX_1,\dots,\XX_N) \ne \emptyset$.
\end{theorem}

\begin{proof}
    Through \cref{lem:mul-par,prop:inv-hom},
    the claim is established by showing $\Pio(\II_1,\dots,\II_N) \ne \emptyset$,
    where $\II_i \coloneqq (I_i, \bar g_i,\lambda)$,
    with $I_i \coloneqq [0,1]$,
    is homomorphic to
    $\XX_i \coloneqq (X_i, g_i, \xi_i)$ via some $\Phi_i \in \para(\XX_i)$.
    Let $(\bar\pi_k)_{k\in\NN}$
    with $\bar\pi_k \in \Pi(\II_1,\dots,\II_N)$ 
    be a minimizing sequence of 
    $\smash{\Fmgw^{\II_1,\dots,\II_N}}$ 
    meaning
    $\lim_{k\to\infty} \smash{\Fmgw^{\II_1,\dots,\II_N}(\bar\pi_k)} 
    = \MGW(\II_1,\dots,\II_N)$.
    Since the measures $\bar\pi_k \in \pmeas([0,1]^N)$
    defined on a compact support
    are tight,
    there exists a weakly convergence subsequence $(\bar\pi_{k_\ell})_{\ell \in \NN}$.
    Furthermore,
    using the linearity of the integral
    and projecting the 
    measure $\bar\pi_{k_\ell}$ onto the respective marginals, 
    we obtain
    \begin{align*}
        &\Fmgw^{\II_1,\dots,\II_N}(\bar\pi_{k_\ell})
        \\
        &= \frac{1}{2} \sum_{i,j=1}^N \rho_i \rho_j
        \iint_{(I_i \times I_j)^2} \lvert g_i(x_i,x'_i) - g_j(x_j,x'_j) \rvert^2
        \dx (P_{I_i \times I_j})_\# \bar \pi_{k_\ell}(x_i,x_j)
        \dx (P_{I_i \times I_j})_\# \bar \pi_{k_\ell}(x'_i,x'_j).
    \end{align*}
    The integrals over the bi-marginal plans on the right-hand side are weakly continuous 
    \cite[Lem~5.5]{sturm2020space}.
    Thus, $\smash{\Fmgw^{\II_1,\dots,\II_N}}$ is weakly continuous,
    and the weak limit of $(\bar\pi_{k_\ell})_{\ell \in \NN}$ is a minimizer of \cref{eq:mgw-fun}.
\end{proof}

\section{GW Tangent Space}\label{sec:4}

The space $(\GM,\GW)$ forms a Riemannian orbifold \cite{sturm2006geometry}. 
The reference contains an extensive study of the geometry of $(\GM,\GW)$,
where the tangent space is constructed in the sense of Alexandrov
by considering the completion of the set of geodesic directions
with respect to a specific metric.
Ultimately, the tangent space $\T_{\Yf}$ at $\Yf \in \GM$
can be defined as the set of equivalence classes:
\begin{equation*}
    \T_{\Yf} 
    \coloneqq 
    \biggl(
    \bigcup_{ (Y,h,\upsilon) \in \Yf} \gauges(Y,\upsilon) 
    \biggr)
    \biggm/ \sim,
\end{equation*}
where the union is taken over all representatives 
$(Y,h,\upsilon)$ of $\Yf$. 
Furthermore, 
$f \in \gauges(Y,\upsilon)$ and 
$\tilde{f} \in \gauges(\tilde{Y},\tilde{\upsilon})$
related to the representatives 
$\YY \coloneqq (Y,h,\upsilon)$ 
and $\tilde \YY \coloneqq (\tilde{Y},\tilde{h},\tilde{\upsilon})$
are equivalent, i.e.\ $f \simeq \tilde{f}$, 
if there exists an optimal plan 
$\tau \in \Pio(\YY,\tilde \YY)$
such that $f(\cdot_1,\cdot_3) = \tilde{f}(\cdot_2,\cdot_4)$ 
almost everywhere with respect to
$\tau(\cdot_1,\cdot_2) \otimes \tau(\cdot_3,\cdot_4)$
on $(Y \times \tilde Y)^2$.
Similarly to gm-spaces,
the equivalence class of $f$ is denoted by 
$\ff \coloneqq \llbracket f \rrbracket$.
Because any tangent representative $f \in \ff$ depends
on the space representative 
$\YY \coloneqq (Y,h,\upsilon)$ on which it is defined, 
we introduce the notation $f \rest_{\YY}$
whenever we want to emphasize this dependence.
The space $\T_{\Yf}$ becomes a complete metric space 
with the tangent metric:
\begin{equation}
    \label{eq:tan-met}
    d_{\T_{\Yf}}
    (\ff, \tilde \ff) 
    \coloneqq
    \inf \bigl\{
    \|f(\cdot_1,\cdot_3) - \tilde f(\cdot_2, \cdot_4)\|
    _{L^2_\sym(\tau(\cdot_1,\cdot_2) \otimes \tau(\cdot_3,\cdot_4))} 
    :
    \tau \in \Pio(\YY,\tilde{\YY})
    \bigr\}
\end{equation}
for any $f \rest_{\YY} \in \ff$
and $\tilde f \rest_{\tilde \YY} \in \tilde\ff$.
Note that the tangent metric is independent of 
the chosen representatives.
The infimum in \cref{eq:tan-met}
is always attained. 
Indeed, the objective can
be written as $\smash{\Fgw^{\ZZ,\tilde{\ZZ}}}$
with $\ZZ = (Y,f,\upsilon)$, 
$\tilde{\ZZ} = (\tilde Y,\tilde{f},\tilde \upsilon)$.
Leveraging a similar argument as in \cref{thm:min-exists}, 
the minimizer can be constructed as a weak limit
of elements in $\Pio(\YY,\tilde{\YY})$.
Using \cite[Lem.~5.5]{sturm2020space},
we can show that this limit 
is element of $\Pio(\YY,\tilde{\YY})$ as well,
cf.\ the proof of \cref{prop:sim_eq_continuous}.

The \emph{exponential map} 
$\Exp_{\Yf} \colon \T_{\Yf} \to \GM$
is defined by
\begin{equation*}
    \Exp_{\Yf} (\ff) 
    \coloneqq 
    \llbracket (Y,h + f, \upsilon) \rrbracket
    \quad\text{for any}\quad
    f\rest_{(Y,h,\upsilon)} \in \ff.
\end{equation*}
Conversely, the \emph{logarithmic map} $\Log_{\Yf}$
mapping $\GM$ into a subset of $\T_{\Yf}$ is given by
\begin{equation*}
    \Log_{\Yf}(\Xf)
    \coloneqq
    \bigl\{
    \llbracket (g-h) \rest_{(Y \times X, h, \pi)} \rrbracket
    :
    \YY \coloneqq (Y, h, \upsilon) \in \Yf,
    \XX \coloneqq (X,g,\xi) \in \Xf,
    \pi \in \Pio(\YY,\XX)
    \bigr\}.
\end{equation*}
The transition from one representative to another representative 
in the definition of the logarithmic map
can be described using optimal meltings.

\begin{lemma}\label{lem:equivalence_of_tangents_melting}
    Let $\XX \coloneqq (X, g, \xi)$, 
    $\YY \coloneqq (Y, h, \upsilon)$,
    and $\tilde\YY \coloneqq (\tilde Y, \tilde h, \tilde \upsilon)$
    be gm-spaces with $\YY \simeq \tilde\YY$.
    If $\pi \in \Pi(\YY,\XX)$
    and $\sigma \in \Pio(\YY,\tilde\YY)$,
    then
    \begin{equation*}
        (g - h)\rest_{(Y \times X, h, \pi)} 
        \simeq 
        (g - \tilde{h})\rest_{(\tilde{Y} \times X, \tilde h, \tilde{\pi})}
        \quad\text{for any}\quad
        \tilde{\pi} \in \Melt_\YY(\sigma,\pi).
    \end{equation*}
\end{lemma}

\begin{proof}
    On the basis of
    the underlying gluing $\gamma \in \Gamma_\YY(\sigma,\pi)$
    of the chosen $\tilde{\pi} \in \Melt_\YY(\sigma,\pi)$,
    we consider $\tau \coloneqq (P_Y,P_X,P_{\tilde{Y}},P_X)_\# \gamma$.
    By the push-forward construction,
    $\tau$ is optimal in the sense of
    \[
    \tau \in 
    \Pio((Y \times X, h, \pi),(\tilde{Y} \times X, \tilde{h}, \tilde{\pi})).
    \]
    Moreover,
    a short computation shows
    $
    (g-h)(\cdot_1, \cdot_3) 
    = (g- \tilde h)(\cdot_2, \cdot_4)$
    almost everywhere with respect to
    $\tau(\cdot_1,\cdot_2) \otimes \tau(\cdot_3,\cdot_4)$
    on $(Y \times X \times \tilde Y \times X)^2$
    yielding the desired equivalence.
\end{proof}

If $\pi \in \Pio(\YY,\XX)$,
and thus $\Melt_\YY(\sigma,\pi) \subset \Pio(\tilde\YY,\XX)$
by \cref{lem:melting_remains_optimal},
then \cref{lem:equivalence_of_tangents_melting}
allows to switch between different representatives of $\Yf$  
in the definition of the logarithmic map.
Similarly,
we can transfer the logarithmic map
between different representations of $\Xf$.

\section{Free-Support GW Barycenters}\label{sec:5}

Barycenters are general Fréchet means on arbitrary metric spaces. 
For the GW setting,
let $\Xf_1,\dots,\Xf_N$ be given gm-spaces, 
$\rho \in \Delta_{N-1}$ be given weights.
An associated \emph{GW barycenter} is any solution to
\begin{equation}
    \label{eq:free-supp}
    \GWB_\rho(\Xf_1,\dots,\Xf_N)
    \coloneqq
    \inf_{\Yf \in \GM} \Fgwb^{\Xf_1,\dots,\Xf_N} (\Yf)
\end{equation}
with the GW barycenter loss
\begin{equation}
    \label{eq:gwb-loss}
     \Fgwb^{\Xf_1,\dots,\Xf_N} (\Yf)
     \coloneqq
     \sum_{i=1}^N \rho_i \GW_2^2(\Xf_i,\Yf).
\end{equation}
Since the infimum in \cref{eq:free-supp} is taken
over arbitrary gm-spaces without restriction to their domain,
the minimizer---%
whenever it exists---%
is called a \emph{free-support GW barycenter}.
The functional \cref{eq:gwb-loss} is continuous
with respect to the topology 
induced by the GW distance on $\GM$.
For the special case of mm-spaces
with compact support,
there exists at least one barycenter,
and every barycenter can be characterized
as solution of an appropriate MGW problem
\cite[Thm~5.1]{BBS2022multi}.
This characterization carries over to 
the more general setting of gm-spaces,
where the proof coincides line by line.
For the sake of completeness,
the detailed proof is given in \cref{sec:proof-char-bary}.

\begin{theorem}
    \label{thm:gen-bary}
    Let $\Xf_i \coloneqq \llbracket(X_i,g_i,\xi_i)\rrbracket$,
    $i = 1,\dotsc,N$,
    and $\rho \in \Delta_{N-1}$.
    There exists $\Yf \in \GM$ minimizing \cref{eq:free-supp}.
    Moreover,
    the homomorphic classes solving $\GWB_\rho(\Xf_1,\dots,\Xf_N)$
    are given by
    \begin{equation*}
        \bigl\{
            \bigl\llbracket X_\times, \sum\nolimits_{i=1}^N \rho_i g_i, \hat{\pi}
            \bigr\rrbracket 
            :
            \hat{\pi} \in \Pio^\rho(\XX_1,\dotsc,\XX_N)
        \bigr\}.
    \end{equation*}
\end{theorem}

From a numerical perspective,
the main issue of the free-support barycenter characterization
in \cref{thm:gen-bary}
is the required optimal multi-marginal transport plan,
whose computation for $N \geq 3$ is in general intractable. 
However, 
for Gaussian gm-spaces 
whose gauges are given by the standard scalar product, 
there exists a multi-marginal plan and GW barycenter 
with closed forms 
as the following result shows.
The theorem makes use of the following notation:
we denote the upper left 
$d \times d$ submatrix of any matrix $B$ 
with $B^{(d)}$ 
and the diagonal sign matrices by
\begin{equation*}
    \mathcal{D}_{d} \coloneqq \{\diag(s_1,\dotsc,s_{d}) : s_{k} \in \{-1,1\}\}.
\end{equation*}
The proof is provided in \cref{sec:proof-gaussian}.
\begin{theorem}\label{thm:gaussian}
    Let $d_1 \geq \dotsc \geq d_N \in \NN$.
    Consider the gm-spaces
    $\XX_i = 
    (\R^{d_i},\langle \cdot, \cdot \rangle_{d_i},\xi_i)$, 
    where $\langle \cdot, \cdot \rangle_{d_i}$ 
    is the standard scalar product on $\R^{d_i}$ 
    and $\xi_i = \N(0,\Sigma_i)$ 
    are centered non-degenerate Gaussian distributions 
    on $\R^{d_i}$, $i=1,\dotsc,N$.
    Furthermore, 
    let the covariance matrices $\Sigma_i$ 
    be diagonalized by
    $\Sigma_i = P_i D_i P_i^{\tT}$ 
    where the eigenvalues
    are sorted
    in decreasing order.
    For fixed
    $\tilde{I}_i \in \mathcal{D}_{d_i}$,
    let
    $T_i: \R^{d_1} \to \R^{d_i}$
    be defined by
    \begin{equation}\label{eq:T_gaussian}
    T_i(x)
    = P_i A_i P_1^\tT x, 
    \quad A_i \coloneqq
    \Bigl(
        \tilde{I}_{d_i} 
        D_i^{\nicefrac{1}{2}} 
        \Bigl(D_1^{(d_i)}\Bigr)^{-\nicefrac{1}{2}} 
        \Bigm| 0_{d_i, d_1 - d_i}
    \Bigr)
    \in \R^{d_i \times d_1}.
    \end{equation}
    For all $\rho \in \Delta_{N-1}$, it holds
    \[
    (T_1,\dotsc,T_N)_\# \xi_1 
    \in \Pio^\rho(\XX_1,\dotsc,\XX_N),
    \]
    and
    \begin{equation}\label{eq:gauss_bary}
    \YY \coloneqq (\R^{d_1}, \langle \cdot, \cdot \rangle_{d_1}, \hat{\upsilon}),
    \quad 
    \hat{\upsilon} \coloneqq \N\biggl(0,
    \sum_{i=1}^N \rho_i 
    \underbrace{
    \begin{pmatrix}
        D_i & 0\\
        0 & 0
    \end{pmatrix}
    }_{\in \R^{d_1}}
    \biggr)
    \end{equation}
    solves $\GWB_{\rho}(\llbracket \XX_1 \rrbracket ,\dotsc, \llbracket \XX_N \rrbracket)$.
\end{theorem}

The previous result generalizes both,
\cite[Prop~4.1]{delon2022gromov} 
from the bi-marginal
to the multi-marginal case,
as well as
\cite[Thm.~3]{le2021entropic},
which states that \cref{eq:gauss_bary} is a solution to
a restricted GW barycenter problem, 
where the gauge of the barycenter is a-priori fixed
to be the inner product on Euclidean space.

As mentioned,
in general the computation of optimal
multi-marginal plans is intractable.
For this reason,
we adopt a known procedure to find Fréchet means on manifolds.
The main idea is as follows:
fix a reference point $\Yf \in \GM$ 
and lift $\Xf_1,\dotsc,\Xf_N$ 
into the tangent space $\T_\Yf$ 
via the logarithmic map,
determine a barycenter 
of the lifted gm-spaces in $\T_\Yf$,
project the barycenter back to $\GM$
via the exponential map,
and update $\Yf$.
Iterate the procedure until convergence.
In the following, 
we explore this tangential barycenter procedure in more detail.
Selecting representatives 
$\XX_i = (X_i,g_i,\xi_i) \in \Xf_i$
and $\YY = (Y,h,\upsilon) \in \Yf$, 
fixing optimal plans
$\pi_i \in \Pio(\YY,\XX_i)$, 
and 
defining $\YY_i \coloneqq (Y \times X_i, h, \pi_i)$,
we consider the liftings
\begin{equation}
    \label{eq:lift-Xi}
    \gf_i 
    \coloneqq
    \llbracket (g_i - h)\rest_{\YY_i} \rrbracket 
    \in \Log_{\Yf}(\Xf_i).
\end{equation}
The \emph{tangential barycenter} 
between $\gf_1,\dots, \gf_N$
with respect to 
barycentric coordinates $\rho \in \Delta_{N-1}$
is then the solution to
\begin{equation}
    \label{eq:tan_bary}
    \GWTB_\rho^\Yf(\gf_1,\dots,\gf_N)
    \coloneqq
    \inf_{\ff \in \T_\Yf} 
    \F_{\GWTB^\Yf_\rho}^{\gf_1,\dots,\gf_N} (\ff)
\end{equation}
with the GW tangential barycenter loss
\begin{equation}
    \label{eq:tan-bary-loss}
    \F_{\GWTB^\Yf_\rho}^{\gf_1,\dots,\gf_N} (\ff)
    \coloneqq
    \sum_{i=1}^N \rho_i \, d_{\T_\Yf}^2(\gf_i, \ff).
\end{equation}
For the analysis of the tangential barycenter, 
we define the \emph{set of (multi-marginal) gluings}
between $\pi_i \in \Pi(\YY,\XX_i)$ along $\YY$ by
\begin{equation*}
    \Glue_\YY(\pi_1,\dotsc,\pi_N) 
    \coloneqq
    \Glue_\YY^{\XX_1,\dots,\XX_N}(\pi_1,\dotsc,\pi_N)
    \coloneqq
    \bigl\{
    \gamma \in \p(Y \times X_\times) 
    : 
    (P_{Y \times X_i})_\# \gamma = \pi_i
    \bigr\}
\end{equation*}
and the \emph{set of (multi-marginal) meltings} along $\YY$ by
\begin{equation*}
    \Melt_\YY(\pi_1,\dotsc,\pi_N) 
    \coloneqq
    \Melt_\YY^{\XX_1,\dots,\XX_N}(\pi_1,\dotsc,\pi_N) 
    \coloneqq
    \bigl\{
    (P_{X_\times})_\# \gamma 
    : 
    \gamma \in \Gamma_{\YY}(\pi_1,\dotsc,\pi_N)
    \bigr\}.
\end{equation*}
Similarly as above,
we extend these definitions to arbitrary sets of plans
by taking the union of the elementwise gluings and meltings.
For the sake of brevity,
we rely on the following notation
$\YYtimes[\gamma] \coloneqq (Y \times X_\times, h, \gamma) \in \Yf$ 
for any $\gamma \in \p(Y \times X_\times)$
with $(P_Y)_\# \gamma = \upsilon$
as well as on the mean gauge 
$m_\rho(x_\times,x_\times') \coloneqq \sum_{i=1}^N \rho_i g_i(x_i,x_i')$
which is defined on $X_\times$.

\begin{theorem}\label{thm:tan_min}
    The tangential barycenter problem \cref{eq:tan_bary} admits a solution. 
    Moreover, 
    there exist plans $\pi_i^* \simeq_\YY \pi_i$ 
    so that  
    $\mf_\rho \coloneqq \llbracket (m_\rho - h)\rest_{\YYtimes(\gamma^*)} \rrbracket$ 
    is a tangential barycenter for some gluing 
    $\gamma^* \in \Gamma_\YY(\pi^*_1,\dotsc,\pi^*_N)$.
\end{theorem}

To prove the previous result, we require the following technical Lemma.

\begin{lemma}
    \label{lem:tan-estimate}
    Let $\gf_i$ be defined as in \cref{eq:lift-Xi},
    and let $\ff \in \T_\Yf$ be arbitrary.
    For every $f\rest_{\tilde\YY} \in \ff$
    defined on $\tilde \YY \coloneqq (\tilde Y, \tilde h, \tilde \upsilon) \in \Yf$,
    there exist $\tilde\pi_i \in \Melt_\YY(\Pio(\YY,\tilde\YY), \pi_i)$
    (dependent on $f$)
    so that
    \begin{equation*}
        \F_{\GWTB^\Yf_\rho}^{\gf_1,\dots,\gf_N} (\ff)
        \ge
        \F_{\GWTB^\Yf_\rho}^{\gf_1,\dots,\gf_N} (\tilde\mf)
        \quad\text{with}\quad
        \tilde \mf_\rho 
        \coloneqq 
        \llbracket (m_\rho - \tilde h)\rest_{\tilde\YY_\times[\tilde\gamma]} \rrbracket 
    \end{equation*}
    for all gluings 
    $\tilde\gamma \in \Gamma_{\tilde\YY}(\tilde \pi_1,\dotsc,\tilde \pi_N)$
    and 
    $\tilde\YY_\times[\tilde\gamma] 
    \coloneqq 
    (\tilde Y \times X_\times, \tilde h, \tilde \gamma)$.
\end{lemma}

\begin{proof}
    Let $\tau_i \in \Pio(\tilde\YY,\YY_i)$ realize 
    the tangent metric 
    $\smash{d_{\T_\Yf}^2}(f\rest_{\tilde\YY},(g_i - h)\rest_{\YY_i})$
    in \cref{eq:tan-met}.
    Since
    $\sigma_i 
    \coloneqq 
    (P_{Y \times \tilde Y})_\# \tau_i
    \in \Pio(\YY, \tilde \YY)$,
    the marginal
    $\tilde \pi_i 
    \coloneqq 
    (P_{\tilde Y \times X_i})_\# \tau_i
    \in \Melt_\YY(\sigma_i,\pi_i)$
    is optimal by \cref{lem:melting_remains_optimal}.
    In analogy to above,
    we define 
    $\tilde\YY_i
    \coloneqq 
    (\tilde Y \times X_i, \tilde h, \tilde\pi_i)$.
    Using that
    $h(\cdot_1,\cdot_3) = \tilde h(\cdot_2, \cdot_4)$
    a.e.\ 
    with respect to each
    $\tau_i(\cdot_1,\cdot_2)\otimes \tau_i(\cdot_3,\cdot_4)$,
    we rewrite the tangential barycentric loss in \cref{eq:tan-bary-loss} as
    \begin{align*}
        \F_{\GWTB^\Yf_\rho}^{\gf_1,\dots,\gf_N}(\ff)
        &= \sum_{i=1}^N \rho_i
        \iint_{(\tilde{Y} \times Y \times X_i)^2} 
        \lvert 
        f(\tilde{y},\tilde{y}')
        - g_i(x_i,x_i') 
        + h(y,y')  
        \rvert^2
        \dx \tau_i(\tilde{y},y,x_i) 
        \dx \tau_i(\tilde{y}',y',x_i')
        \\
        &=
        \sum_{i=1}^N \rho_i
        \iint_{(\tilde{Y} \times X_i)^2}  
        \lvert  
        f(\tilde{y},\tilde{y}')
        - g_i(x_i,x_i') 
        + \tilde{h}(\tilde{y},\tilde{y}') 
        \rvert^2 
        \dx \tilde \pi_i(\tilde{y},x_i) 
        \dx \tilde \pi_i(\tilde{y}',x_i')
        \\
        &=
        \iint_{(\tilde{Y} \times X_\times)^2}  
        \sum_{i=1}^N \rho_i
        \lvert  
        f(\tilde{y},\tilde{y}') 
        - g_i(x_i,x_i') 
        + \tilde{h}(\tilde{y},\tilde{y})
        \rvert^2 
        \dx \tilde{\gamma}(\tilde{y},x_\times)
        \dx \tilde{\gamma}(\tilde{y}',x_\times'),
    \end{align*}
    for an arbitrary gluing 
    $\tilde\gamma \in \Gamma_{\tilde\YY}(\tilde \pi_1,\dotsc,\tilde \pi_N)$.
    Exploiting that the functional
    $a \mapsto \sum_{i=1}^N \rho_i \lvert a - a_i \rvert^2$
    is minimized by $a = \sum_{i=1}^N \rho_i a_i$
    and applying this pointwisely 
    inside the integral further yield
    \begin{align*}
        &\F_{\GWTB^\Yf_\rho}^{\gf_1,\dots,\gf_N}(\ff)
        \\
        &\ge
        \iint_{(\tilde{Y} \times X_\times)^2}  
        \sum_{i=1}^N \rho_i
        \lvert
        m_\rho(x_\times,x_\times')
        - \tilde h(\tilde y, \tilde y')
        - g_i(x_i,x_i') 
        + \tilde h(\tilde y, \tilde y')
        \rvert^2 
        \dx \tilde{\gamma}(\tilde{y},x_\times)
        \dx \tilde{\gamma}(\tilde{y}',x_\times')
    \end{align*}
    Writing 
    $Z_\times \coloneqq \tilde Y \times X_\times$
    and
    $Z_i \coloneqq \tilde Y \times X_i$,
    using
    $\tilde{\tau}_i 
    \coloneqq
    (P_{Z_\times},P_{Z_i})_\# 
    \tilde{\gamma}
    \in \Pio(\tilde\YY_\times, \tilde\YY_i)$,
    and taking the infimum over all optimal plans,
    we obtain
    \begin{align*}
        \F_{\GWTB^\Yf_\rho}^{\gf_1,\dots,\gf_N}(\ff)
        &\ge 
        \sum_{i=1}^N \rho_i
        \iint_{
        (Z_\times \times Z_i)^2
        }  
        \lvert 
        (m_\rho - \tilde{h}) (z_1,z_1') 
        - (g_i - \tilde{h}) (z_2,z_2')
        \rvert^2
        \dx 
        \tilde{\tau}_i (z_1,z_2)
        \dx 
        \tilde{\tau}_i (z_1',z_2')
        \\
        &\ge
        \sum_{i=1}^N \rho_i 
        d_{\T_\Yf}^2
        \bigl( 
        (m_\rho - \tilde{h})\rest_{\tilde{\YY}_\times [\tilde\gamma]},
        (g_i - \tilde{h})\rest_{\tilde{\YY}_i}
        \bigr)
        =
        \sum_{i=1}^N \rho_i 
        d_{\T_\Yf}^2
        \bigl( 
        (m_\rho - \tilde{h})\rest_{\tilde{\YY}_\times [\tilde\gamma]},
        (g_i - h)\rest_{\YY_i}
        \bigr),
    \end{align*}
    where the last equality follows from
    \cref{lem:equivalence_of_tangents_melting}
    with $\sigma_i$ from above.
\end{proof}

\begin{corollary}\label{cor:tan-estimate}
    Let $\gf_i$ be defined as in \cref{eq:lift-Xi},
    and let $\ff \in \T_\Yf$ be arbitrary.
    Then there exist $\pi_i^* \simeq_\YY \pi_i$
    (dependent on $\ff$)
    so that
    \begin{equation*}
        \F_{\GWTB^\Yf_\rho}^{\gf_1,\dots,\gf_N} (\ff)
        \ge
        \F_{\GWTB^\Yf_\rho}^{\gf_1,\dots,\gf_N} (\mf_\rho)
        \quad\text{with}\quad
        \mf 
        \coloneqq 
        \llbracket (m_\rho - h)\rest_{\YY_\times[\gamma]} \rrbracket 
    \end{equation*}
    for at least one gluing 
    $\gamma \in \Gamma_{\YY}(\pi_1^*,\dotsc,\pi_N^*)$.
\end{corollary}

\begin{proof}
    Let
    $\tilde{\pi}_1,\dotsc,\tilde{\pi}_N$
    be constructed as in 
    the proof of \cref{lem:tan-estimate}
    and consider a gluing
    $\tilde\gamma \in \Glue_{\tilde\YY}(\tilde \pi_1, \dots, \tilde\pi_N)$.
    For fixed $\sigma \in \Pio(\tilde \YY, \YY)$,
    \cref{lem:equivalence_of_tangents_melting} implies
    $(m_\rho - \tilde h)\rest_{\tilde \YY_\times[\tilde \gamma]}
    \simeq
    (m_\rho - h)\rest_{\YY_\times [\gamma^*]}$
    for $\gamma^* \in \Melt_{\tilde\YY}(\sigma,\tilde\gamma)$.
    Defining $\pi^*_i \coloneqq (P_{Y \times X_i})_\# \gamma^*$,
    we readily obtain
    $\gamma^* \in \Glue_\YY(\pi_1^*, \dots, \pi_N^*)$.
    The assertion $\pi_i^* \simeq_\YY \pi_i$
    now follows from
    \[
    \pi_i^* 
    \in 
    \Melt_\YY(\sigma,\tilde\pi_i)
    \subset \Melt_{\tilde\YY}(\sigma,\Melt_\YY(\sigma_i, \pi_i))
    = \Melt_\YY(\Melt_{\tilde\YY}(\sigma_i^\tT, \sigma), \pi_i)
    \]
    since 
    $\Melt_{\tilde\YY}(\sigma_i^\tT, \sigma)
    \subset \Pio(\YY,\YY)$ by \cref{lem:melting_remains_optimal}.
\end{proof}

The proof of \cref{thm:tan_min} also requires
the following continuity result,
which is a multi-marginal variant of 
\cite[Lem~5.5]{sturm2020space}.

\begin{lemma}\label{lem:multi-cont}
    Let $I_\times \coloneqq \bigtimes_{i=1}^M I_i$ 
    be the unit cube with $I_i \coloneqq [0,1]$.
    For fixed $f_i \in \gauges(I_i, \lambda)$,
    the mapping 
    \begin{equation*}
        \Xi(\eta) 
        \coloneqq 
        \Bigl(
        \iint_{I_\times^2}
        \Bigl\lvert \sum_{i=1}^M f_i(x_i,x_i') \Bigr\rvert^2
        \dx \eta(x_1,\dots,x_M) \dx \eta(x_1',\dots,x_M')
        \Bigr)^{\frac12}
    \end{equation*}
    is continuous on 
    $\Beta_M \coloneqq \{ \eta \in \pmeas(I_\times) : (P_{I_i})_\# \eta = \lambda \}$
    equipped with the weak convergence from $\pmeas(I_\times)$.
\end{lemma}

\begin{proof}
    Similarly to the proof of \cite[Lem~5.5]{sturm2020space},
    for every $\epsilon > 0$,
    there exists a symmetric, continuous function 
    $f_{i,\epsilon} \in C_\sym(I_i \times I_i)$
    so that
    $\smash{\lVert f_i - f_{i,\epsilon} \rVert_{L^2_\sym(\lambda\otimes\lambda)}} 
    \le \epsilon$.
    For these functions,
    we analogously define
    $\Xi_\epsilon(\eta)
    \coloneqq
    \lVert \sum_{i=1}^M f_{i,\epsilon} \rVert_{L^2_\sym(\eta \otimes \eta)}$.
    Fixing $\eta \in \Beta_M$,
    and using the triangle inequality twice,
    we have
    \begin{equation*}
        \lvert \Xi(\eta) - \Xi_\epsilon(\eta) \rvert
        \le 
        \Bigl\lVert 
        \sum_{i=1}^M f_i 
        - \sum_{i=1}^M f_{i,\epsilon} 
        \Bigr\rVert_{L^2_\sym(\eta \otimes \eta)}
        \le
        \sum_{i=1}^M
        \lVert 
        f_i - f_{i,\epsilon} 
        \rVert_{L^2_\sym(\lambda\otimes\lambda)}
        \le M \epsilon.
    \end{equation*}
    Let $\eta_n \weakly \eta$ be a weakly convergence sequence in $\Beta_M$.
    Due to \cite[Lem.~A.1]{BBS2022multi},
    we have
    $\eta_n \otimes \eta_n \weakly \eta \otimes \eta$
    and thus 
    $\Xi_\epsilon(\eta_n) \to \Xi_\epsilon(\eta)$.
    In total,
    we obtain
    \begin{align*}
        \lim_{n \to \infty} 
        \lvert \Xi(\eta_n) - \Xi(\eta) \rvert 
        \leq 
        \lim_{n \to \infty}
        \bigl(
        \lvert \Xi(\eta_n) - \Xi_\epsilon(\eta_n) \rvert
        +
        \lvert \Xi_\epsilon(\eta_n) - \Xi_\epsilon(\eta) \rvert
        +
        \lvert \Xi_\epsilon(\eta) - \Xi(\eta) \rvert
        \bigr)
        \leq 2M \epsilon.
    \end{align*}
    Since $\epsilon > 0$ is arbitrary, 
    the continuity is established.
\end{proof}

\begin{proof}[Proof of \cref{thm:tan_min}]
    Without loss of generality,
    we consider the representatives 
    $\XX_i \coloneqq ([0,1], g_i, \lambda)$,
    $\YY \coloneqq ([0,1], h, \lambda)$,
    and $\tilde \YY \coloneqq ([0,1], \tilde h, \lambda)$.
    Let $(\ff^{(n)})_{n \in \NN} \subset \T_{\Yf}$ be a minimizing sequence 
    of the tangent barycenter functional
    \smash{$\F_{\GWTB^\Yf_\rho}^{\gf_1,\dots,\gf_N}$} 
    in \cref{eq:tan-bary-loss}.
    Due to \cref{cor:tan-estimate},
    for every tangent $\ff^{(n)}$,
    there exists an 
    $\mf^{(n)}_\rho \coloneqq \llbracket (m_\rho-h) \rest_{\YY_\times[\gamma_n]}\rrbracket$
    with
    $\gamma_n \in \Gamma_\YY(\pi^*_{1,n},\dotsc,\pi^*_{N,n})$, $\pi^*_{i,n} \simeq_\YY \pi_i$,
    such that
    \begin{equation*}
        \F_{\GWTB^\Yf_\rho}^{\gf_1,\dots,\gf_N} (\ff^{(n)})
        \ge
        \F_{\GWTB^\Yf_\rho}^{\gf_1,\dots,\gf_N} (\mf^{(n)}_\rho);
    \end{equation*}
    so $(\mf^{(n)}_\rho)_{n\in\NN} \subset \T_\Yf$ is itself a minimizing sequence.
    Now let 
    $\tau_{i,n} \!
    \in \Pio(\YY_i, \YY_\times[\gamma_n]) 
    \subset \pmeas([0,1]^{N+3})$
    realize $\smash{d^2_{\T_\Yf}(\gf_i,\mf^{(n)}_\rho)}$.
    Through the compact support,
    each sequence $(\tau_{i,n})_{n\in\NN}$ is tight,
    and we successively find subsequences 
    such that  
    $\tau_{i,n_\ell} \weakly \tau_i^*$
    and $\gamma_{n_\ell} \weakly \gamma^*$,
    $\ell \to \infty$. 
    Therefore $\tau_i^* \in \Pi(\YY_i, \YY_\times[\gamma^*])$.
    We denote the marginals
    $\pi_i^* \coloneqq (P_{Y \times X_i})_\# \gamma^* = \lim_{n \to \infty} \pi^*_{i,n}$.
    Here, the last equality follows by 
    the weak continuity 
    of the marginal projections.
    Furthermore, 
    due to the weak continuity of 
    the relation $\simeq_\YY$, i.e.\ 
    \cref{prop:sim_eq_continuous},
    we obtain
    $\gamma^* \in \Gamma_\YY(\pi_1^*,\dotsc,\pi_N^*)$
    with
    $\pi_i^* \simeq_\YY \pi_i$.
    Since the gauges of
    $\YY_\times[\gamma_n]$, $n \in \NN$ 
    and $\YY_\times[\gamma^*]$ 
    coincide,
    \cref{lem:multi-cont} moreover implies 
    $\tau_i^* \in \Pio(\YY_i, \YY_\times[\gamma^*])$
    with $\GW(\YY_i, \YY_\times[\gamma^*]) = 0$.
    Making use of \cref{lem:multi-cont} once again,
    we finally have
    \begin{align*}
        &\lim_{\ell\to\infty}
        \F_{\GWTB^\Yf_\rho}^{\gf_1,\dots,\gf_N}(\mf^{(n_\ell)}_\rho)
        \\
        &= 
        \lim_{\ell \to \infty}
        \sum_{i=1}^N \rho_i  
        \iint_{([0,1]^{N+3})^2} 
        \lvert (g_i - h)(\cdot_1,\cdot_3) - (m_\rho - h)(\cdot_2,\cdot_4) \rvert^2 
        \dx \tau_{i,n}(\cdot_1,\cdot_2)
        \dx \tau_{i,n}(\cdot_3,\cdot_4)
        \\
        &=
        \sum_{i=1}^N \rho_i 
        \iint_{([0,1]^{N+3})^2}
        \lvert (g_i - h)(\cdot_1,\cdot_3) - (m_\rho - h)(\cdot_2,\cdot_4) \rvert^2 
        \dx \tau_{i}^*(\cdot_1,\cdot_2)
        \dx \tau_{i}^*(\cdot_3,\cdot_4)
        \\
        &\ge \F_{\GWTB^\Yf_\rho}^{\gf_1,\dots,\gf_N}(\mf^*_\rho)
        \quad\text{with}\quad
        \mf^*_\rho 
        \coloneqq 
        \llbracket (m_\rho-h) \rest_{\YY_\times[\gamma^*]} \rrbracket.
    \end{align*}
    Since $(\mf^{(n_\ell)}_\rho)_{\ell\in\NN}$ is a minimizing sequence,
    $\mf^*_\rho$ has to be a minimizer 
    and has the desired form.
\end{proof}

\cref{thm:tan_min} motivates the following Fréchet mean procedure:
\begin{enumerate}
    \item Choose representatives $\XX_i \in \Xf_i$.
    \item Select a starting reference $\Yf \coloneqq \llbracket \YY \rrbracket \in \GM$.
    \item Compute $\pi_i \in \Pio(\YY,\XX_i)$ 
    to determine the liftings $\gf_i$ in \cref{eq:lift-Xi}.
    \item Find a tangential barycenter $\mf_\rho$ as in \cref{thm:tan_min}.
    \item Update $\Yf \leftarrow \Exp_\Yf(\mf_\rho)$,
    and repeat from step 3. 
\end{enumerate}
Note that the mean gauge of 
$\Exp_\Yf(\mf_\rho) = \llbracket( Y \times X_\times, m_\rho, \gamma^*)\rrbracket$
only depends on the subspace $X_\times$
so that
$\Exp_\Yf(\mf_\rho) = \llbracket(X_\times, m_\rho, (P_{X_\times})_\#\gamma^*)\rrbracket$.
At its heart,
this barycenter procedure is a fixpoint iteration
of the set-valued mapping
\begin{equation*}
    \TB_\rho(\Yf) 
    \coloneqq
    \bigl\{
        \llbracket
        (X_\times, m_\rho , (P_{X_\times})_\# \gamma^*) 
        \rrbracket
        :
        \gamma^* \in \Glue_\YY(\pi_1^*, \dots, \pi_N^*)
        \;\text{as in \cref{thm:tan_min}},
        \YY \in \Yf
    \bigr\},
\end{equation*}
i.e.\ we consider the iteration $\Yf_{n+1} \in \TB_\rho(\Yf_n)$
beginning from some $\Yf_0 \in \GM$.

\section{Tangential Fixpoint Iteration}\label{sec:6}

Unfortunately,
the proposed tangential barycenter procedure is intractable in practice.
In particular,
the search for the minimizer $\mf_\rho$ in \cref{thm:tan_min} is unmanageable
due to the lack of efficient methods to determine
the gluing $\gamma^*$ numerically.
As a remedy,
we relax the fixpoint iteration regarding $\TB_\rho$
and refrain from the exact minimization on $\T_\Yf$.
More precisely,
we study the relaxed, set-valued mapping
\begin{equation*}
    \tildeTB_\rho(\Yf)
    \coloneqq
    \bigl\{
    \llbracket (X_\times, m_\rho, \mu) \rrbracket
    :
    \mu \in \Melto_\Yf(\XX_1, \dots, \XX_N)
    \bigr\}
\end{equation*}
with the \emph{set of optimal meltings}
\begin{equation*}
    \Melto_\Yf(\XX_1,\dots, \XX_N)
    \coloneqq
    \bigl\{
    \mu \in \Melt_\YY(\pi_1,\dots,\pi_N)
    :
    \pi_i \in \Pio(\YY, \XX_i)
    \bigr\}.
\end{equation*}
for any $\YY \in \Yf$.
Note that $\Melto_\Yf$ is independent 
of the chosen representative.

\begin{proposition}
    The mapping $\tildeTB_\rho$ 
    (as well as $\TB_\rho$) 
    is well-defined 
    and independent from the selected representatives
    $\XX_1,\dots,\XX_N, \YY$.
\end{proposition}

\begin{proof}
    Initially,
    let the representatives $\XX_i \coloneqq (X_i, g_i, \xi_i) \in \Xf$
    be fixed
    and homomorphic to $\II_i = ([0,1], \bar g_i, \lambda)$
    via $\Phi_i \in \para(\XX_i)$.
    Similarly to the proof of \cref{prop:inv-hom},
    for two representatives 
    $\YY \coloneqq (Y,h,\upsilon),
    \tilde\YY \coloneqq (\tilde Y, \tilde h, \tilde \upsilon) 
    \in \Yf$,
    we find parametrizations 
    $\Psi \in \para(\YY)$ and $\tilde\Psi \in \para(\tilde\YY)$
    so that
    $\YY$ and $\tilde\YY$ are homomorphic to a common $\JJ = ([0,1], \bar h, \lambda)$
    via $\Psi$ and $\tilde\Psi$.
    On account of \cref{lem:bi-par},
    for each $\pi_i \in \Pio(\YY,\XX_i)$,
    there exists at least one $\bar\pi_i \in \Pio(\JJ,\II_i)$
    such that 
    $\pi_i = (\Psi \times \Phi_i)_\# \bar\pi_i$
    and 
    $\tilde \pi_i 
    \coloneqq (\tilde\Psi \times \Phi_i)_\# \bar\pi_i
    \in \Pio(\tilde \YY, \XX_i)$.
    In analogy to \cref{lem:mul-par},
    for these plans,
    we have 
    \begin{align*}
        \Glue_\YY(\pi_1,\dots,\pi_N) 
        &= 
        (\Psi \times \Phi_1 \times \cdots \times \Phi_N)_\#
        \Glue_\JJ(\bar\pi_1, \dots, \bar\pi_N),
        \\
        \Glue_{\tilde\YY}(\tilde\pi_1,\dots,\tilde\pi_N) 
        &= 
        (\tilde\Psi \times \Phi_1 \times \cdots \times \Phi_N)_\#
        \Glue_\JJ(\bar\pi_1, \dots, \bar\pi_N).
    \end{align*}
    Determining the marginals with respect to $X_\times$ and $I_\times$,
    we thus deduce
    \begin{equation*}
        \Melt_\YY(\pi_1,\dots,\pi_N) 
        = 
        (\Phi_1 \times \cdots \times \Phi_N)_\#
        \Melt_\JJ(\bar\pi_1, \dots, \bar\pi_N)
        =
        \Melt_{\tilde\YY}(\tilde\pi_1,\dots,\tilde\pi_N).
    \end{equation*}
    Taking the union over all optimal plans,
    which are induced by $\Pio(\JJ,\II_i)$,
    we obtain the independence of $\Melto_\Yf$
    from the actual representative $\YY \in \Yf$.
    A similar argument for the representatives of $\Xf_i$
    yields that $\smash{\tildeTB_\rho}$ generates the same homomorphic classes
    for any $\XX_i \in \Xf$.
\end{proof}

Different from the tangential barycenter procedure 
in the previous section,
the elements of $\smash{\tildeTB_\rho}(\Yf)$ are simply based on
the $N$ optimal GW plans and their melting.
For this reason,
the fixpoint iteration 
$\Yf_{n+1} \in \smash{\tildeTB_\rho}(\Yf_n)$
becomes numerically tractable.
Despite the relaxation,
this iteration yields a sequence 
with non-increasing barycentric loss,
which can be seen as an analogue 
to the Wasserstein case 
in \cite[Prop.~3.3]{ABC16fixedpoint}.

\begin{theorem}\label{thm:G_monotone}
    Let $\Xf_1,\dots,\Xf_N,\Yf \in \GM$ and 
    $\rho \in \Delta_{N-1}$
    be fixed.
    Then every $\Zf \in \smash{\tildeTB_\rho}(\Yf)$ satisfies
    \begin{equation*}
    \Fgwb^{\Xf_1,\dots,\Xf_N}(\Yf) 
    \geq 
    \Fgwb^{\Xf_1,\dots,\Xf_N}(\Zf) 
    +
    \GW_2^2(\Yf,\Zf).
    \end{equation*}
    In particular, 
    every barycenter $\Yf$ is a fixpoint of $\tildeTB_\rho$,
    i.e.\ $\Zf = \Yf$.
\end{theorem}

Remarkably, the previous result holds
for all 
$\Zf \in \tildeTB_\rho(\Yf)$
where only the gauge of $\Zf$ depends on $\rho$.
Hence, for any arbitray fixed
$\mu \in \Melto_\Yf(\XX_1, \dots, \XX_N)$
and $\Zf 
= \llbracket (X_\times, m_\rho, \mu) \rrbracket$,
the above inequality holds for all
$\rho \in \Delta_{N-1}$.

\begin{proof}
    We consider the representatives 
    $\XX_i \coloneqq (X_i, g_i, \xi_i) \in \Xf_i$
    and 
    $\YY \coloneqq (Y, h, \upsilon) \in \Yf$.
    Since $\Zf \in \smash{\tildeTB_\rho}(\Yf)$,
    there exist 
    $\pi_i \in \Pio(\YY, \XX_i)$
    and
    $\gamma \in \Glue_\YY(\pi_1,\dots,\pi_N)$
    such that
    $\ZZ \coloneqq (X_\times, m_\rho, \mu) \in \Zf$
    with $m_\rho \coloneqq \sum_{i=1}^N g_i$
    and $\mu \coloneqq (P_{X_\times})_\# \gamma$.
    The barycentric loss of $\Yf$ may be written as
    \begin{align*}
        \Fgwb^{\Xf_1,\dots,\Xf_N}(\Yf) 
        &= 
        \sum_{i=1}^N \rho_i 
        \iint_{(Y \times X_i)^2} 
        \lvert h(y,y') - g_i(x_i,x_i') \rvert^2 
        \dx \pi_i(y, x_i) 
        \dx \pi_i(y', x_i') 
        \\
        &= 
        \iint_{(Y \times X^\times)^2}
        \sum_{i=1}^N \rho_i 
        \lvert h(y,y') - g_i(x_i,x_i') \rvert^2 
        \dx \gamma(y, x_\times)
        \dx \gamma(y', x_\times').
    \end{align*}
    Due to
    $\sum_{i=1}^N \rho_i \lvert b - a_i \rvert^2 
    = \lvert b - a \rvert^2 +
    \sum_{i=1}^N \rho_i \lvert a - a_i \rvert^2$
    for $a \coloneqq \sum_{i=1}^N \rho_i a_i$,
    we obtain
    \begin{align*}
        \Fgwb^{\Xf_1,\dots,\Xf_N}(\Yf) 
        &= \iint_{(Y \times X^\times)^2}
        \lvert h(y,y') - m_\rho(x_\times, x_\times') \rvert^2 
        \dx \gamma(y,x_\times) \dx \gamma(y',x_\times')
        \\
        &\quad+
        \sum_{i=1}^N \rho_i
        \iint_{X_\times^2} 
        \lvert m_\rho(x_\times, x_\times') - g_i(x_i,x_i') \rvert^2 
        \dx \mu (x_\times) 
        \dx \mu (x_\times').
    \end{align*}
    We estimate the terms on the right-hand side separably.
    Since $\gamma \in \Pi(\YY, \ZZ)$, 
    it clearly holds
    \begin{equation*}
        \iint_{(Y \times X^\times)^2}
        \lvert h(y,y') - m_\rho(x_\times, x_\times') \rvert^2 
        \dx \gamma(y,x_\times) \dx \gamma(y',x_\times')
        \ge 
        \GW_2^2(\YY,\ZZ)
        =
        \GW_2^2(\Yf,\Zf).
    \end{equation*}
    Defining $\tilde\pi_i \coloneqq (\id, P_{X_i})_\# \mu \in \Pi(\ZZ,\XX_i)$,
    we further estimate the second term by
    \begin{align*}
        &\sum_{i=1}^N \rho_i
        \iint_{X_\times^2} 
        \lvert m_\rho(x_\times, x_\times') - g_i(x_i,x_i') \rvert^2 
        \dx \mu (x_\times) 
        \dx \mu (x_\times')
        \\
        &=
        \sum_{i=1}^N \rho_i 
        \iint_{(X_\times \times X_i)^2} 
        \lvert m_\rho(x_\times, x_\times') - g_i(\tilde x_i,\tilde x_i') \rvert^2 
        \dx \tilde\pi_i (x_\times, \tilde{x}_i)
        \dx \tilde\pi_i (x_\times', \tilde{x}_i')
        \\
        &\ge 
        \sum_{i=1}^N \rho_i \GW_2^2(\ZZ,\XX_i)
        =
        \sum_{i=1}^N \rho_i \GW_2^2(\Zf,\Xf_i)
        = 
        \Fgwb^{\Xf_1,\dots,\Xf_N}(\Zf),
    \end{align*}
    which establishes the assertion.
\end{proof}

Furthermore,
every sequence iteratively generated 
by applying $\tildeTB_\rho$ contains
a convergent sub\-sequence,
whose limit can be interpreted as fixpoint of $\smash{\tildeTB_\rho}$.

\begin{theorem}\label{thm:G_fixpoint}
    Let $\Xf_1,\dots,\Xf_N \in \GM$
    and $\rho \in \Delta_{N-1}$ be fixed.
    Then every subsequence of $(\Yf_n)_{n \in \NN}$
    with $\Yf_{n+1} \in \smash{\tildeTB_\rho(\Yf_n)}$
    contains a convergent subsequence $\Yf_{n_\ell} \to \Yf \in \GM$
    so that 
    $\Yf \in \smash{\tildeTB_\rho}(\Yf)$.
\end{theorem}

To give the proof, 
we require the following two lemmata
about the continuity and stability 
of the GW transport on the box.
For this,
let $I_\times \coloneqq \times_{i=1}^M I_i$ be the unit cube 
with $I_i \coloneqq [0,1]$,
and consider the measures with uniform marginals given by
$\Beta_M \coloneqq \{ \eta \in \pmeas(I_\times) : (P_{I_i})_\# \eta = \lambda \}$.
We equip $\Beta_M$ with the weak convergence from $\pmeas(I_\times)$.

\begin{lemma}\label{lem:continuity_on_box}
    For fixed $h_i \in \gauges(I_i,\lambda)$,
    consider $\Yf_n \coloneqq \llbracket(I_\times, h, \upsilon_n)\rrbracket$
    with $h \coloneqq \sum_{i=1}^M h_i$
    and $\upsilon_n \in \Beta_M$. 
    If $\upsilon_n \weakly \upsilon \in \Beta_M$, 
    then 
    $\Yf_n \to \Yf 
    \coloneqq \llbracket(I_\times, h, \upsilon)\rrbracket 
    \in \GM$ as $n \to \infty$.
\end{lemma}

\begin{proof}
    Consider the representatives 
    $\YY_n \coloneqq (I_\times, h, \upsilon_n)$
    and
    $\YY \coloneqq (I_\times, h, \upsilon)$.
    As a direct consequence of \cite[Thm.~5.20]{villani2008optimal}, 
    for every subsequence of $(\YY_n)_{n\in\NN}$,
    we can find a further subsequence $(\YY_{n_\ell})_{\ell\in\NN}$
    and plans $\pi_{n_\ell} \in \Pi(\YY_{n_\ell},\YY)$ 
    such that 
    $\pi_{n_\ell} \weakly \pi \coloneqq (\id,\id)_\# \upsilon$. 
    [The plans can be constructed
    by considering the Wasserstein distance on $\pmeas(I_\times)$.]
    Since the gauges of $\YY_{n_\ell}$ and $\YY$ coincide,
    both GW functionals \smash{$\Fgw^{\smash{\YY_{n_\ell}},\YY}$}
    and \smash{$\Fgw^{\YY,\YY}$} become
    $\Xi(\bullet) 
    \coloneqq 
    \lVert 
    h(\cdot_1,\cdot_3) - h(\cdot_2,\cdot_4) 
    \rVert_{L^2_\sym(\bullet \otimes \bullet)}$.
    As $\pi_{n_\ell} \in \Beta_{2M}$,
    \cref{lem:multi-cont} implies
    \begin{equation*}
    \GW(\YY_{n_\ell},\YY)
    \leq \Xi(\pi_{n_\ell}) 
    \to \Xi(\pi) = 0
    \quad\text{as}\quad
    \ell\to\infty.
    \end{equation*}
    Since this holds true for all subsequences,
    we obtain the assertion as desired.
\end{proof}

To distinguish the different boxes,
we define 
$I_{1,\times} \coloneqq \bigtimes_{i=1}^{M_1} I_i$
and $I_{2,\times} \coloneqq \bigtimes_{i=1}^{M_2} I_i$
with $I_i \coloneqq [0,1]$.

\begin{lemma}\label{lem:stability_on_box}
    Consider
    $\YY_n \coloneqq (I_{1,\times}, h, \upsilon_n)$
    with $h \coloneqq \sum_{i=1}^{M_1} h_i$,
    $h_i \in \gauges(I_i, \lambda)$,
    and $\upsilon_n \in \Beta_{M_1}$,
    and let $\XX = (I_{2,\times},g,\xi)$
    be given
    with \smash{$g \coloneqq \sum_{i=1}^{M_2} g_i$},
    $g_i \in \gauges(I_i, \lambda)$,
    and $\xi \in \Beta_{M_2}$.
    Furthermore,
    let $\pi_n \in \Pio(\YY_n,\XX)$,
    and assume $\pi_n \weakly \pi \in \Beta_{M_1+M_2}$.
    Then
    $\pi$ is optimal between 
    $\YY \coloneqq (I_{1,\times},h,\upsilon)$
    with $\upsilon \coloneqq (P_{I_{1,\times}})_\# \pi$
    and $\XX$,
    i.e.\
    $\pi \in \Pio(\YY,\XX)$.
\end{lemma}

\begin{proof}
    Defining
    $\Xi(\eta) 
    \coloneqq 
    \lVert h - g \rVert_{L^2_\sym(\eta \otimes \eta)}$
    for all $\eta \in \Beta_{M_1 + M_2}$,
    we notice that $\Xi$ coincides 
    with the GW functional \smash{$\Fgw^{\YY_n,\XX}$}
    on $\Pi(\YY_n,\XX) \subset \Beta_{M_1+M_2}$
    and with \smash{$\Fgw^{\YY,\XX}$}
    on $\Pi(\YY,\XX) \subset \Beta_{M_1+M_2}$.
    Hence,
    we have
    \begin{equation}\label{eq:proof_1_stability_on_box}
        \lvert \Fgw^{\YY,\XX}(\pi) - \GW(\YY,\XX) \rvert 
        \leq 
        \lvert \Xi(\pi) - \Xi(\pi_n) \rvert
        +
        \lvert \Xi(\pi_n) - \GW(\YY,\XX)\rvert.
    \end{equation}
    By \cref{lem:multi-cont}, 
    the first term converges to zero.
    Furthermore,
    the weak convergence of $\pi_n$ implies
    $\upsilon_n \weakly \upsilon$,
    which yields 
    $\llbracket \YY_n \rrbracket 
    \to \llbracket \YY \rrbracket
    \in \GM$
    due to \cref{lem:continuity_on_box}. 
    Since $\pi_n$ is optimal,
    we have $\Xi(\pi_n) = \GW(\YY_n,\XX)$;
    so the second term 
    in \cref{eq:proof_1_stability_on_box} 
    converges to zero
    by the continuity of the metric.
    This shows $\Fgw^{\YY,\XX}(\pi) = \GW(\YY,\XX)$ 
    and thus $\pi \in \Pio(\YY,\XX)$ as desired.
\end{proof}

\begin{proof}[Proof of \cref{thm:G_fixpoint}]
    Without loss of generality,
    we consider the representatives 
    $\XX_i \coloneqq (I_i,g_i,\lambda)$,
    where $I_i \coloneqq [0,1]$.
    As before,
    the related box is denoted by 
    $I_\times \coloneqq \bigtimes_{i=1}^N I_i$,
    and the mean gauge as
    $m_\rho(x_\times,x_\times') \coloneqq \sum_{i=1}^N \rho_i g_i(x_i, x_i')$.
    Except for the starting space $\Yf_1$,
    the elements of $(\Yf_n)_{n \in \NN}$
    are defined as
    $\Yf_n = \llbracket \YY_n \rrbracket$
    with $\YY_n \coloneqq (I_\times, m_\rho , \mu_n)$
    for some melting $\mu_n \in \Melt_{\Yf_{n-1}}(\XX_1,\dots,\XX_N) 
    \subset \pmeas(I_\times)$.
    Let $\gamma_n$, $n \in \NN$,
    be the associated gluings.
    For $n > 2$,
    we have
    $\gamma_n \in \pmeas(I_\times \times I_\times)$
    with first marginal $\mu_{n-1}$ and 
    second marginal $\mu_n$.
    Since $I_\times \times I_\times$ is compact,
    the sequence $\gamma_n$ is tight,
    and any subsequence contains 
    a weakly convergent subsequence $\gamma_{n_\ell}$
    with some limit $\gamma \in \pmeas(I_\times \times I_\times)$.
    Due to \cref{lem:continuity_on_box},
    we have $\Yf_{n_\ell} \to \tilde \Yf \coloneqq \llbracket\tilde\YY\rrbracket$
    with $\tilde\YY \coloneqq (I_\times, m, \tilde \mu)$,
    where $\tilde\mu$ is the second marginal of $\gamma$,
    and $\Yf_{n_\ell - 1} \to \Yf \coloneqq \llbracket\YY\rrbracket$
    with $\YY \coloneqq (I_\times, m, \mu)$,
    where $\mu$ denotes the first marginal of $\gamma$.
    Since $(P_{I_\times \times I_i})_\# \gamma_{n_\ell} 
    \in \Pio(\YY_{n_\ell - 1}, \XX_i)$,
    \cref{lem:stability_on_box} ensures the optimality of
    $(P_{I_\times \times I_i})_\# \gamma 
    \in \Pio(\YY, \XX_i)$;
    so we have $\tilde\mu \in \Melt_{\Yf}(\XX_1,\dots,\XX_N)$
    and $\tilde\Yf \in \smash{\tildeTB_\rho(\Yf)}$.
    Furthermore,
    \cref{thm:G_monotone} yields
    \begin{equation*}
        \Fgwb^{\Xf_1,\dots,\Xf_N}(\Yf_{n_\ell-1}) 
        \geq 
        \Fgwb^{\Xf_1,\dots,\Xf_N}(\Yf_{n_\ell}) 
        + \GW_2^2(\Yf_{n_\ell-1}, \Yf_{n_\ell}).
    \end{equation*}
    Taking the limit over $\ell$,
    and exploiting that
    $(\Fgwb(\Yf_n))_{n \in \NN}$ is monotonically decreasing 
    and bounded from below,
    we obtain $\GW(\Yf, \tilde\Yf) = 0$,
    establishing the assertion.
\end{proof}

\section{Algorithm and Practical Usage}\label{sec:7}
\label{sec:in_practice}

\subsection{Explicit Implementation of the 
Tangential Barycenter Method}

The numerical fixpoint iteration 
with respect to $\tildeTB_\rho$ 
is summarized in \cref{alg:1}.
Comments on the GW step 
and on the gluing-melting step can be found below.
\begin{algorithm}
	\begin{algorithmic}
	    \State \textbf{Input:} 
	    \parbox[t]{300pt}{gm-spaces $\XX_i \coloneqq (X_i,g_i,\xi_i)$, $i=1,\dots,N$,
        \\
        barycentric coordinates $\rho \in \Delta_{N-1}$,\\
	    initial gm-space $\YY \coloneqq (Y,h,\upsilon)$
        }
	    \While{not converged}
            \State \emph{GW Step:} 
            Compute optimal GW plans 
            $\hat{\pi}_i$ between $\YY$ and $\XX_i$,
            $i = 1,\dotsc,N$.
            \State \emph{Gluing-Melting Step:}
            Construct a gluing 
            $\gamma \in 
            \Gamma_\YY
            (\hat{\pi}_1,\dotsc,\hat{\pi}_N)$ 
            with melting 
            $\mu \coloneqq P_{X_\times \#} \gamma$.
            \State Set 
            $\YY \gets (X_\times, \sum_{i=1}^N \rho_i g_i, \mu)$.
	    \EndWhile
     	\State \textbf{Output:} Approximate barycenter $\YY = (X_\times, \sum_{i=1}^N \rho_i g_i, \mu)$
	\end{algorithmic}
	\caption{Tangential Barycenter Algorithm}
	\label{alg:1}
\end{algorithm}
Choosing for instance
$\YY = \XX_{1}$
as the initial gm-space
slightly reduces the computation of the algorithm 
in the first iteration
since we may simply set 
$\hat{\pi}_1 \coloneqq (\id,\id)_\# \xi_1$
instead of computing the plan.
Moreover, 
if we are in the case $N=2$, 
for any computed $\hat{\pi}_2 \in \Pio(\YY,\XX_2)$
it holds
\[
\Gamma_{\YY}((\id,\id)_\# \xi_1,\hat{\pi}_2) 
= \{(P_{Y},P_Y,P_{X_2})_\# \pi_2\},
\quad
\text{ which yields }
\quad
\Melt_{\XX_{1}}((\id,\id)_\# \xi_1, \hat{\pi}_2) 
= \{\hat{\pi}_2\}.
\]
Thus, after the first iteration $\YY$ has the form
\begin{equation*}
\YY = 
\bigl(X_1 \times X_2, 
\rho_1 g_1 + \rho_2 g_2, 
\hat{\pi}_2\bigr),
\end{equation*}
which is already a solution to $\GWB_{(\rho_1,\rho_2)}(\XX_1,\XX_2)$,
cf.\ \cref{thm:gen-bary}.
In other words, \cref{alg:1} 
reaches optimality already after one iteration. 
The same holds true for
$N$ gm-spaces given as
centered Gaussian distributions 
endowed with the standard Euclidean scalar product.
Indeed, 
consider the setting from \cref{thm:gaussian}, 
i.e. let $d_1 \geq \dotsc \geq d_N$ and
$\XX_i = 
(\R^{d_i},\langle \cdot, \cdot \rangle_{d_i},\xi_i)$, 
with $\xi_i = \N(0,\Sigma_i)$.
Taking $\YY = \XX_1$,
we may set
$\hat{\pi}_i = (\id,T_i)_\# \xi_1 \in \Pio(\XX_1,\XX_i)$
in the GW step,
where $T_i$ is defined as in \cref{thm:gaussian}.
This yields
\begin{equation*}
    \Gamma_\YY(\pi_1,\dotsc,\pi_N)
    =
    \{(\id,T_1,\dotsc,T_N)_\# \xi_1\}
    \quad
    \text{so that}
    \quad
    \Melt_\YY(\pi_1,\dotsc,\pi_N)
    =
    \{(T_1,\dotsc,T_N)_\# \xi_1\}.
\end{equation*}
Due to \cref{thm:gaussian}
we have 
$(T_1,\dotsc,T_N)_\# \xi_1 
\in \Pio^\rho(\XX_1,\dotsc,\XX_N)$,
so that optimality is reached 
after one iteration.

In the following, 
we provide some further insight on
the practical application of \cref{alg:1},
discuss the recovery of embeddings
of the barycenters
and show how an approximation 
of pairwise distances
of the input spaces can be obtained.

\paragraph{The GW Step}
At the heart of the iterations in \cref{alg:1}
are the GW computations between $\YY = (Y,h,\upsilon)$ and $\XX_i = (X_i,g_i,\xi_i)$,
hence we comment on existing methods 
to compute $\GW$ which we use
in our numerical section below.
For a bigger picture, see e.g.\ \cite{PCS2016,xu2019scalable}.
One way to numerically approximate solutions 
of $\GW(\YY,\XX_i)$
is by performing a block-coordinate descent 
of the bi-convex relaxed problem
\begin{equation}\label{eq:biconvex_relax}
\inf_{\pi,\gamma \in \Pi(\upsilon,\xi_i)}
\|h(\cdot_1,\cdot_3) - g_i(\cdot_2,\cdot_4)\|^2
    _{L^2_\sym(\pi(\cdot_1,\cdot_2) \otimes \gamma(\cdot_3,\cdot_4))}.
\end{equation}
Fixing, say, $\gamma$ and minimizing with respect to $\pi$,
we recover the Kantorovich problem
\begin{equation}\label{eq:local_problem}
\inf_{\pi \in \Pi(\upsilon,\xi_i)} \int_{Y \times X_i}
\|h(y,\cdot_3) - g_i(x,\cdot_4)\|^2
    _{L^2_\sym(\gamma(\cdot_3,\cdot_4))} \dx \pi(y,x).
\end{equation}
Most practical algorithms for $\GW$ rely on this structure,
i.e.\ for a fixed initial guess 
$\gamma \in \Pi(\upsilon,\xi_i)$ 
solve \cref{eq:local_problem}, 
then set $\gamma = \pi$ 
and repeat this iteration 
until a convergence criterion is satisfied.
The local problem may be solved directly,
e.g.\ with the network simplex algorithm.
Another option is to regularize \cref{eq:local_problem}
by adding an entropic regularizer like
$+ \eps \KL(\pi,\upsilon \otimes \xi_i)$,
or
$+ \eps \KL(\pi,\gamma)$
to the functional,
where $\KL$ denotes the Kullback--Leibler divergence.
In both cases, 
this allows to leverage 
the Sinkhorn algorithm \cite{C2013} 
for solving the local problem 
at the cost of 
the usual entropic blurring.

\paragraph{The Gluing-Melting Step}
Let $\YY \coloneqq (Y,h,\upsilon)$, 
$\XX_i \coloneqq (X_i,g_i,\xi_i)$ 
be gm-spaces
with support sizes $m,n_i \in \NN$, $i = 1,\dotsc,N$.
Solutions that are obtained in the above way
without regularization are naturally very sparse.
More precisely, there exists
a solution 
$\pi_i \in \Pi(\upsilon,\xi_i)$
to \cref{eq:local_problem} with
\[
\lvert \supp(\pi_i) \rvert \leq m + n_i - 1 \ll n_i m,
\]
see e.g.\ \cite[Thm.~1]{friesecke2023gencol}.
In an effort to obtain a sparse gluing and
keep computational complexity low,
we propose to use the \emph{north-west corner rule},
see e.g.\ \cite[Sec.~3.4.2]{PC19book} for details.
This generates a gluing 
$\gamma \in \Gamma_{\YY}(\pi_1,\dotsc,\pi_N)$ 
in up to $m + \sum_{i=1}^N n_i$ operations 
and ensures 
$\lvert \supp(\gamma) \rvert 
\leq m + \bigl(\sum_{i=1}^N n_i\bigr) - N$.
Using this construction, 
after $k$ iterations of \cref{alg:1},
the support size of $\YY$ can be bounded by
\[
m + k\biggl(\sum_{i=1}^N n_i\biggr) - kN.
\]
In contrast, for any approximate solution $\pi_i$ 
of the regularized version of \cref{eq:local_problem}, 
it holds $\lvert \supp(\pi_i) \rvert = n_i m$, 
$i =1,\dotsc,N$.
In this case
$\lvert \supp (\gamma) \rvert = m \prod_{i=1}^N n_i$
and thus $\lvert \supp (\mu) \rvert = \prod_{i=1}^N n_i$
for any gluing $\gamma \in \Gamma_{\YY}(\pi_1,\dotsc,\pi_N)$
and associated melting $\mu \in \Melt_{\YY}(\pi_1,\dotsc,\pi_N)$
which practically inhibits 
the use entropic regularizers.
However, if additionally 
$\lvert X_i \rvert = m$
and
$\xi_i$ is the uniform distribution on $X_i$,
$i =1,\dotsc,N$,
we propose the following alternative construction 
for the gluing-melting step.
First, we initialize \cref{alg:1} 
with $\YY = (Y,h,\upsilon)$ such that
$\lvert Y \rvert = m$.
After solving the GW step 
with entropic regularization
to obtain measures $\pi_i \in \Pi(\YY,\XX_i)$,
we do the following.
Iterating over $\ell = 1,\dotsc,m$, we recursively construct
\[
T_i(x^{(i)}_\ell) = \argmax \left\{\frac{\pi_i(\{(y,x^{i}_\ell)\})}{\nu(\{y\})} : y \in Y,
T_i(x^{(i)}_{\tilde{\ell}}) \neq y,\, \tilde{\ell} = 1,\dotsc,\ell-1 \right\}, \quad i =1,\dotsc,N.
\]
By construction, 
the obtained maps $T_i: X_i \to Y$ are invertible,
so we may replace the gluing-melting step by instead setting
\[
\tilde{\gamma} = (\id,T_1^{-1},\dotsc,T_N^{-1})_\# \upsilon
\quad \text{and} \quad \tilde{\mu} = P_{X_\times \#} \tilde{\gamma}.
\]
Note that $\tilde{\gamma}$ and $\tilde{\mu}$
does not have to be an actual gluing
and melting, respectively.
The construction ensures 
$\lvert \supp(\tilde{\mu}) \rvert = m$ so that
we may readily apply this procedure again
in the next iteration.
Figuratively,
the multi-marginal plan $\tilde{\mu}$
inscribes one-to-one correspondences 
between the inputs
by associating their points, 
if most of their respective masses is transported 
to the same point in $\YY$.
In our practical application we refer to this procedure as the
as a \emph{maximum rule without replacement}.
Since $\xi_i$ is the uniform distribution on $m$ points 
for all $i=1,\dotsc,N$, 
we obtain
$\tilde{\mu} 
\in \Pi(\XX_1,\dotsc,\XX_N)$.
The maximum rule without replacement is particularly well suited when we are interested
in exact one-to-one correspondences of the given inputs.

\subsection{Embeddings for Barycenters}\label{subsec:embedding}

For certain practical tasks,
an embedding of the $\GW$ interpolants may be required.
However, finding an appropriate embedding 
can be a very challenging task on its own and is 
a common drawback for existing methods
such as the state-of-the-art algorithm 
provided by Peyré, Cuturi and Solomon \cite{PCS2016}.
Here, the output of the function 
takes the form of a dissimilarity matrix
which is usually embedded using some
numerically expensive distance-preserving embedding technique 
like multi-dimensional scaling.
A remedy to this has been proposed in \cite{BBS2022multi}, 
where the support of the seeked barycenter is fixed a-priori.
In the following, we want to give some insights 
on this embeddding-issue in relation to our proposed method
and give some special cases where (high-dimensional) embeddings 
of tangent barycenters can be obtained naturally.
To this end, 
consider the inputs
$\XX_i = (X_i,g_i,\xi_i)$, $i =1,\dotsc,N$,
and let $(X_\times,m_\rho,\mu)$
with $\mu \in \Pi(\XX_1,\dotsc,\XX_N)$
be (an approximation of) the barycenter.
\paragraph{Euclidean Space}
First, we consider the Euclidean case, 
i.e.\ let $X_i = \R^{d_i}$, 
$d_i \in \NN$, $i =1,\dotsc,N$,
and set 
$d_{\oplus} \coloneqq \sum_{i=1}^N d_i$.
\begin{itemize}
    \item[i)]
    Let $g_i \coloneqq \|\cdot - \cdot \|_{d_i,1}$
    be the $1$-norm on $\R^{d_i}$, $i=1,\dotsc,N$.
    Then
    \begin{align*}
    m_\rho(x_\times,x_\times') 
    &= \sum_{i=1}^N \rho_i g_i(x_i,x_i')
    = \sum_{i=1}^N \rho_i\|x_i - x_i'\|_{d_i,1}
    = \sum_{i=1}^N \|\rho_i x_i - \rho_i x_i' \|_{d_i,1}
    \\
    &= \| S_\rho(x_\times) - S_\rho(x_\times') \|_{d_\oplus,1},
    \end{align*}
    where $S_\rho(x_\times) \coloneqq (\rho_1 x_1,\dotsc,\rho_N x_N)$, 
    $x_\times \in X_\times$.
    Thus,
    \[
    (X_\times, m_\rho, \mu) 
    \simeq (\R^{d_\oplus}, \|\cdot - \cdot \|_{d_\oplus,1}, (S_{\rho})_\# \mu).
    \]
    \item[ii)]
    Either let $g_i = \|\cdot - \cdot\|_{d_i,2}^2$ 
    be the squared standard Euclidean norm on $\R^{d_i}$ 
    or $g_i = \langle \cdot, \cdot \rangle_{d_i}$ 
    be the standard scalar product on $\R^{d_i}$.
    Then we obtain
    \begin{align*}
    m_\rho(x_\times,x_\times') 
    &= \sum_{i=1}^N \rho_i g_i(x_i,x_i')
    = \sum_{i=1}^N  g_i(\sqrt{\rho_i} x_i , \sqrt{\rho_i} x_i')
    \\
    &= 
    \begin{cases}
        \|S_{\sqrt{\rho}}(x_\times) - S_{\sqrt{\rho}}(x_\times')\|_{d_\oplus,2}^2
        &\text{if } g_i = \|\cdot - \cdot \|_{d_i,2}^2, \quad i=1\dotsc,N
        \\
        \bigl\langle 
        S_{\sqrt{\rho}}(x_\times),S_{\sqrt{\rho}}(x_\times')
        \bigr\rangle_{d_\oplus}
        &\text{if } g_i = \langle \cdot, \cdot \rangle_{d_i}, \quad i=1\dotsc,N,
    \end{cases}
    \end{align*}
    where $S_{\sqrt{\rho}}(x_\times)
    \coloneqq 
    (\sqrt{\rho_1} x_1,\dotsc,\sqrt{\rho_N} x_N)$, 
    $x_\times \in X_\times$.
    Thus,
    \[
    (X_\times, m_\rho, \mu) 
    \simeq 
    \begin{cases}
        (\R^{d_\oplus}, \|\cdot - \cdot \|_{d_\oplus,2}^2, (S_{\sqrt{\rho}})_\# \mu)
        &\text{if } g_i = \|\cdot - \cdot \|_{d_i,2}^2, \quad i=1,\dotsc,N
        \\
        (\R^{d_\oplus}, \langle \cdot,\cdot\rangle_{d_\oplus}, (S_{\sqrt{\rho}})_\# \mu)
        &\text{if } g_i = \langle \cdot,\cdot\rangle_{d_i}, \quad i=1,\dotsc,N.
    \end{cases}
    \]
\end{itemize}
For the above cases,
our method provides us with embeddings
without applying any distance-preserving embedding technique.
However, it should be noted that these embeddings will naturally be high-dimensional.
To reduce the dimensionality of the embedding, 
we may apply computationally inexpensive
techniques such as the principle component analysis (PCA).
In the special cases, 
where $\mu = (T_1,\dotsc,T_N) _\# \xi_1$
for isometries (or more generally linear functions) 
$T_i:X_1 \to X_i$, $i=1,\dotsc,N$,
the gm-space $(X_\times,m_\rho,\mu)$ 
is isomorphic to the (rescaled) graph of the linear function 
$(T_1,\dotsc,T_N):X_1 \to \R^{d_\oplus}$,
which is already a $d_1$-dimensional subspace of $\R^{d_\oplus}$.
Notably, this gives a further motivation 
for the linear Gromov--Monge problem \cite[Def.~4.2.4]{Vayer2020},
which naturally restricts the support of plans to lie
on the graph of linear functions.

\paragraph{Surfaces}
In the practical context,
shapes are usually 2d surfaces that are 
embedded in the 3d Euclidean space.
In mathematical terms this leads us to Riemannian manifolds.
Let $M \subset \R^{3}$ be a smooth connected 2d manifold.
The \emph{Riemannian distance} on $M$ is defined by 
\[
d_M(x,y) 
\coloneqq 
\inf \Bigl\{
\int_{a}^b \| \gamma'(t) \|_{3,2} \dx t
: \gamma \in C^1([a,b],M), \gamma(a) = x, \gamma(b) = y
\Bigr\}.
\]
Fortunately, Riemannian manifolds are isometric in the Riemannian sense
if and only they are isometric in the metric sense with respect to their
Riemannian distance \cite{lee2018introduction}.
Let $M_i \subset \R^3$ be 2d Riemannian manifolds
with Riemannian distances $d_{M_i}$ 
and probability measures $\xi_i \in \p(M_i)$,
$i=1,\dotsc,N$.
Consider the gm-spaces 
$\XX_i \coloneqq (M_i,d_{M_i}^2,\xi_i)$.
Using the fact that
\smash{$\rho_i d_{M_i}^2(x_i,x_i') 
= d_{\sqrt{\rho_i} M_i}^2(\sqrt{\rho_i} x_i, \sqrt{\rho_i}x_i')$}
and following similar arguments as in ii) of the previous paragraph, 
we can show that
\[
(M_\times, m_\rho, \mu)
\simeq 
\bigl(S_{\sqrt{\rho}}(M_\times), 
\tilde{d}_\rho, 
(S_{\sqrt{\rho}})_\# \mu\bigr),
\]
where 
\smash{$M_\times \coloneqq \bigtimes_{i=1}^N M_i$},
\smash{$m_\rho \coloneqq \sum_{i=1}^N \rho_i d_{M_i}^2$},
and 
\smash{$\tilde{d}_\rho \coloneqq \sum_{i=1}^N d_{\sqrt{\rho_i} M_i}^2$}.
The Cartesian product
\smash{$S_{\sqrt{\rho}}(M_\times)$}
is a ($2N$)-dimensional Riemannian manifold embedded in $\R^{3N}$. 
Furthermore, 
\smash{$\tilde{d}_\rho$}
is the squared Riemannian distance 
on $S_{\sqrt{\rho}}(M_\times)$.
Indeed, 
let $\gamma 
= (\gamma_1,\dotsc,\gamma_N)$
be an element of
$C^1([a,b],S_{\sqrt{\rho}}(M_\times))$, 
then
\begin{equation*}
    \int_a^b \| \gamma'(t) \|_{3N,2}^2 \dx t  
    =
    \int_a^b \sum_{i=1}^N \| \gamma_i'(t) \|_{3,2}^2 \dx t
    = \sum_{i=1}^N \int_a^b \| \gamma_i'(t) \|_{3,2}^2 \dx t.
\end{equation*}
Therefore, 
\smash{$\gamma \in C^1([a,b],S_{\sqrt{\rho}}(M_\times))$}
realizes \smash{$d_{S_{\sqrt{\rho}}(M_\times)}$}
if and only if all of its components $\gamma_i$ 
realize $d_{\sqrt{\rho_i} M_i}$, $i=1,\dotsc,N$
so that \smash{$\tilde{d}_\rho = d_{S_{\sqrt{\rho}}(M_\times)}$}. 
Therefore, we may interpret the gm-space $(M_\times, m_\rho, \mu)$
as a $(2N)$-dimensional Riemannian (product) manifold in $\R^{3N}$.
If $\mu = (T_1,\dotsc,T_N)_\# \xi_1$ 
for Riemmanian isometries $T_i :M_1 \to M_i$, $i=1,\dotsc,N$,
we know that $\supp(\mu) \subset S_{\sqrt{\rho}}(M_\times)$
admits a 2d isometric representative (e.g.\ $M_1$), 
however, in contrast to the Euclidean case, 
it is unclear if $\supp(\mu)$ 
naturally lies on a lower-dimensional subspace of $\R^{3N}$.

\subsection{Approximation of Pairwise Distances}
\label{subsec:tangent_distance_approximation}

In \cite{beier2022linear}, 
the authors consider an
approximation of pairwise GW distances 
of a set of spaces 
$\Xf_i 
= \llbracket \XX_i \rrbracket
= \llbracket (X_i,g_i,\xi_i) \rrbracket$, 
$i =1,\dotsc,N$,
by fixing a reference point 
$\Yf 
= \llbracket \YY \rrbracket 
= \llbracket (Y,h,\upsilon) \rrbracket$,
and proposing the
\emph{linear Gromov--Wasserstein distance}
\[
\LGW_\Yf(\Xf_i,\Xf_j) 
\coloneqq
\inf_{
\substack{
\ff_i \in \Log_{\Yf}(\Xf_i)
\\
\ff_j \in \Log_{\Yf}(\Xf_j)
}}
d_{\T_\Yf}(\ff_i,\ff_j)
= 
\hspace{-10pt}
\inf_{
\substack{
\mu_{i,j} \in \Melt_{\YY}(\pi_i,\pi_j)
\\
\pi_i \in \Pio(\upsilon, \xi_i), \,
\pi_j \in \Pio(\upsilon, \xi_j)
}
}
\hspace{-10pt}
\Fgw^{\XX_i,\XX_j}(\mu_{i,j}),
\]
where the second equality has been shown in the reference.
The proposed distance follows the same idea
as the so-called linear Wasserstein distance
\cite{wang2013linear}:
instead of directly computing the GW distance,
the input spaces $\Xf_i$, $i=1,\dotsc,N$
are lifted into the tangent space of the a-priori
fixed reference space $\Yf$ via the logarithmic map,
where subsequent distance computations are carried out.
Because the logarithmic map is a set-valued function
and thus may yield multiple tangents, 
and because the evaluated tangent distance
$d_{\T_\Yf}(\ff_i,\ff_j)$
is an
upper bound to $\GW(\XX_i,\XX_j)$,
the quantity is minimized 
with respect to all tangents in
$\Log_{\Yf}(\Xf_i)$ and $\Log_{\Yf}(\Xf_j)$.
Effectively, 
this results in minimizing with respect to
all possible optimal $\GW$ plans 
from $\YY$ to $\XX_i$ and $\XX_j$.
In practice, 
this minimization is intractable
and is usually replaced by fixing
one computed plan $\hat{\pi}_i$ per input $\XX_i$,
$i =1,\dotsc,N$.
After each iteration of \cref{alg:1},
we obtain an approximation of $\LGW_\Yf(\Xf_i,\Xf_j)$
by evaluating 
\smash{$\Fgw^{\XX_i,\XX_j}((P_{X_i \times X_j})_\# \mu)$},
$i,j = 1,\dotsc,N$,
where 
$\mu \in \Melt_{\YY}(\hat{\pi}_1,\dotsc,\hat{\pi}_N)$
is the measure obtained 
from the gluing-melting step.
We explore this approximation 
in our numerical examples
in further detail.

\section{Numerical Experiments}\label{sec:8}

We employ \cref{alg:1}
together with our remarks of the previous section
to interpolate and classify
3d shapes
as well as for multi-graph matching 
on a protein network dataset.
All experiments are run on an off-the-shelf
MacBook Pro (Apple M1 chip, 8~GB RAM) 
and are implemented%
\footnote{The source code 
is publicly available at 
\url{https://github.com/robertbeinert/tangential-GW-barycenter}.} 
in Python 3. 
We partly rely on the 
Python Optimal Transport (POT) toolbox 
\cite{POT-toolbox},
networkX \cite{networkx}
for handling graphs,
Trimesh \cite{trimesh},
as well as some publicized 
implementations of \cite{xu2019scalable}.

\subsection{3D Shape Interpolations}\label{subsec:interpolation}

In the following, we use \cref{alg:1} 
paired with the considerations from \cref{sec:in_practice}
to generate qualitative GW interpolations 
between multiple 3d surface meshes.
We focus on two datasets, 
namely the mesh deformation dataset \cite{mesh3d_animals}, which predominantly comprises 
animals in various poses.
Furthermore, we consider the training subset of
\cite{faust}, which consists of 3d scans of humans
in various poses.
The meshes of the latter consist of 6890 nodes.
Since some of the meshes 
from the mesh deformation dataset 
comprise up to 42k nodes, 
we use a quadric mesh simplification 
provided by the Trimesh package to reduce
them to support sizes between 5000 and 6000.

From a surface mesh with nodes $V \subset \R^3$ 
and (triangular) faces $F \subset V^3$, 
we first extract a weighted graph $G= (V,E)$,
where the edge set $E$ is obtained by incorporating
edges between any nodes which share a common face.
The weight of the edge is given by 
the Euclidean distance between its respective nodes.
From the graph $(V,E)$,
we proceed to extract a gm-space $\XX = (X,g,\xi)$ 
as follows:
The node set serves as the support of $\XX$, i.e.\ $X = V$, 
the gauge $g$ is is taken 
as the weighted Djikstra distances 
between the individual nodes
and $\xi$ is the uniform distribution.

We utilize \cref{alg:1},
to compute approximate barycenters between various choices
of inputs $\XX_i = (X_i,g_i,\xi_i)$, $i=1,\dotsc,N$, 
for either $N=2$ or $N=4$.
In all cases that follow, we set the initial guess 
to be one of the inputs, i.e.\ we initialize
the algorithm with $\YY = \XX_1$.
As previously discussed,
this reduces the computational complexity slightly.
Moreover, 
to obtain the entire interpolation for $N=2$ inputs, 
we only require one iteration of \cref{alg:1}.
More precisely, 
in this case,
we merely compute an 
approximate optimal plan $\pi \in \Pi(\xi_1,\xi_2)$
via block-coordinate descent of the bi-convex relaxation. 
The barycenters then admit the form
\begin{equation}\label{eq:bary_geodesic_2}
\YY_\rho = (X_1 \times X_2, \rho_1 g_1 + \rho_2 g_2, \pi).
\end{equation}
where $\rho_1 + \rho_2 = 1$.
For $N = 4$,
we restrict ourselves to
3 iterations of \cref{alg:1}
which is indicated to be sufficient 
as we observed fast convergence
in terms of the barycentric loss
$\Fgwb^{\XX_1,\dotsc,\XX_N}$.
As with $N=2$, 
for the GW step we apply 
block-coordinate descent.
For the gluing-melting step,
we use the north-west-corner-rule,
see \cref{sec:in_practice}.
For all but the first iteration,
we initialize 
the block-coordinate descent 
in the GW step between 
$\YY = (X_\times,m_\rho,\mu)$ and $\XX_i$
with the plan 
$(P_{X_\times},P_{X_i})_\# \mu 
\in \Pi(\YY,\XX_i)$.
Since the algorithm only requires the barycentric coordinates $\rho$
for the mean gauge $m_\rho$ 
in the final step
and not in the main iteration
to calculate the multi-marginal plan,
we obtain approximate barycenters 
simultaneously for all $\rho \in \Delta_3$.
The corresponding outputs
admit the form
\begin{equation}\label{eq:bary_geodesic_4}
\YY_\rho = \Bigl(\bigtimes\nolimits_{i=1}^4 X_i, \sum\nolimits_{i=1}^4 \rho_i g_i, \mu\Bigr).
\end{equation}
We remark that for these experiments,
all obtained approximate GW transport plans
were supported on the graph of a transport map, 
so that
$\lvert \supp(\pi) \rvert 
= \lvert \supp(\mu) \rvert 
= 6890$.
To construct appropriate surface meshes,
we fix an index $i_0 \in \{1,\dotsc,N\}$
and incorporate a face between any three points 
$(x_1,\dotsc,x_N), 
(y_1,\dotsc,y_N), 
(z_1,\dotsc,z_N)$
in the support of the (approximate) barycenter
whenever $(x_{i_0},y_{i_0},z_{i_0})$ is a face
in the $i_0$-th input.
To plot the obtained meshes in 3d,
we adopt the idea of
\cref{subsec:embedding} 
for the Euclidean case.
For this,
we replace the geodesic distances $g_i$ in 
\cref{eq:bary_geodesic_2} 
and \cref{eq:bary_geodesic_4}
with the squared Euclidean norm,
i.e.\ we consider gm-spaces of the form
\[
(\R^{3N}, 
\sum_{i=1}^N \rho_i \|\cdot - \cdot\|_{3,2}^2,
\alpha) 
\simeq (\R^{3N}, \|\cdot - \cdot\|_{3N,2}^2, (S_{\sqrt{\rho}})_\#\alpha),
\quad
\alpha \coloneqq
\begin{cases}
    \pi & \text{ for } N = 2,\\
    \mu & \text{ for } N = 4
\end{cases},
\]
with 
$S_{\sqrt{\rho}}(x_\times) 
= (\sqrt{\rho_1} x_1,\dotsc,\sqrt{\rho_N} x_N)$.
The support of $(S_{\sqrt{\rho}})_\#\alpha$
together with the previously constructed faces
creates a $3N$-dimensional surface mesh.
Using a PCA,
we compute an associated 3d affine subspace 
with the smallest projection error.
Finally,
we obtain 3d meshes by projecting
onto the found subspace.
Note that the replacement of the gauge
together with the dimension reduction
is merely a heuristic for plotting
purposes.

\begin{figure}
    \centering
    \includegraphics[
    width=0.9\linewidth]
    {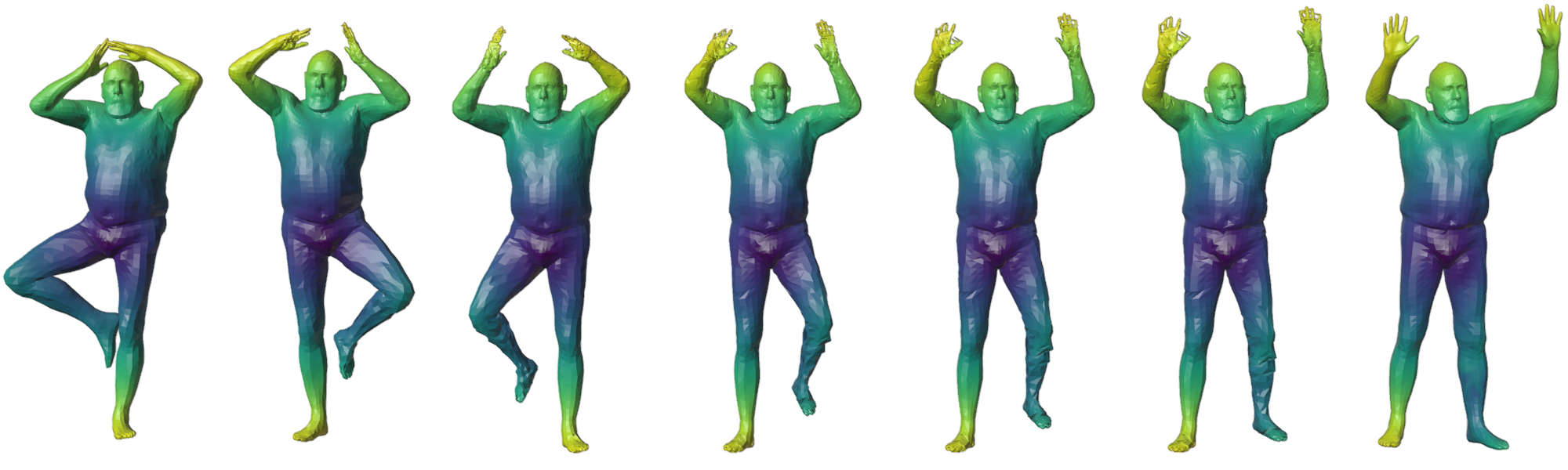}
    \\
    \setlength{\labwidth}{50pt}
    \parbox{\labwidth}{\centering 
    \small{\hspace{-10pt}diam: \, 1.82 \\ \hspace{-10pt}PCA: \, ---}
    }
    \parbox{\labwidth}{\centering 
    \small{1.86 \\ 0.15}
    }
    \parbox{\labwidth}{\centering 
    \small{1.90 \\ 0.20}
    }
    \parbox{\labwidth}{\centering 
    \small{1.94 \\ 0.20}
    }
    \parbox{\labwidth}{\centering 
    \small{1.98 \\ 0.17}
    }
    \parbox{\labwidth}{\centering 
    \small{2.02 \\ 0.13} 
    }
    \parbox{\labwidth}{\centering 
    \small{2.06 \\ ---}
    }
    \\
    \includegraphics[
    width=0.9\linewidth]{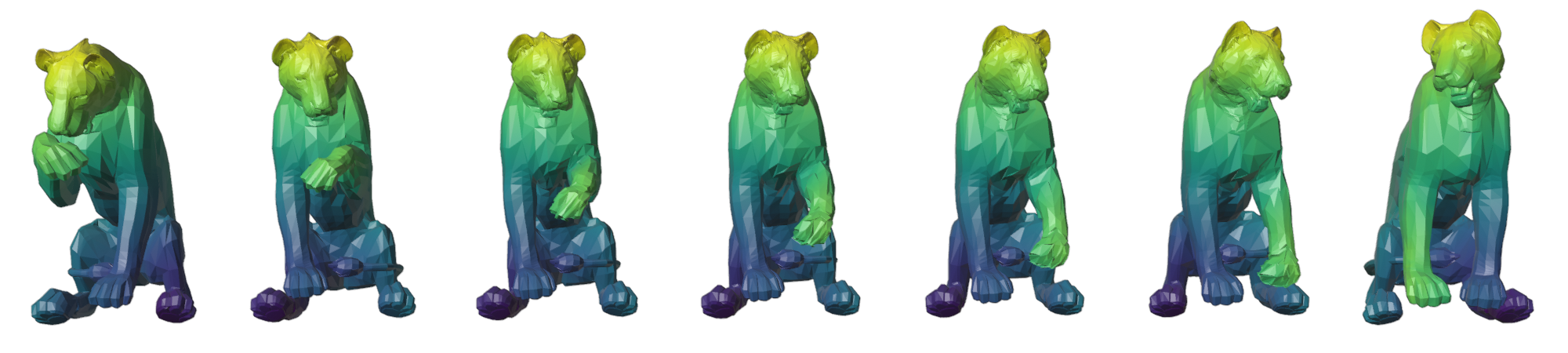}
    \\
    \setlength{\labwidth}{50pt}
    \parbox{\labwidth}{\centering 
    \small{\hspace{-10pt}diam: \, 0.61 \\ \hspace{-10pt}PCA: \, ---}
    }
    \parbox{\labwidth}{\centering 
    \small{0.61 \\ 0.02}
    }
    \parbox{\labwidth}{\centering 
    \small{0.61 \\ 0.03}
    }
    \parbox{\labwidth}{\centering 
    \small{0.62 \\ 0.04}
    }
    \parbox{\labwidth}{\centering 
    \small{0.62 \\ 0.04}
    }
    \parbox{\labwidth}{\centering 
    \small{0.62 \\ 0.03} 
    }
    \parbox{\labwidth}{\centering 
    \small{0.63 \\ ---}
    }
    \\[10pt]
    \includegraphics[
    width=0.9\linewidth]
    {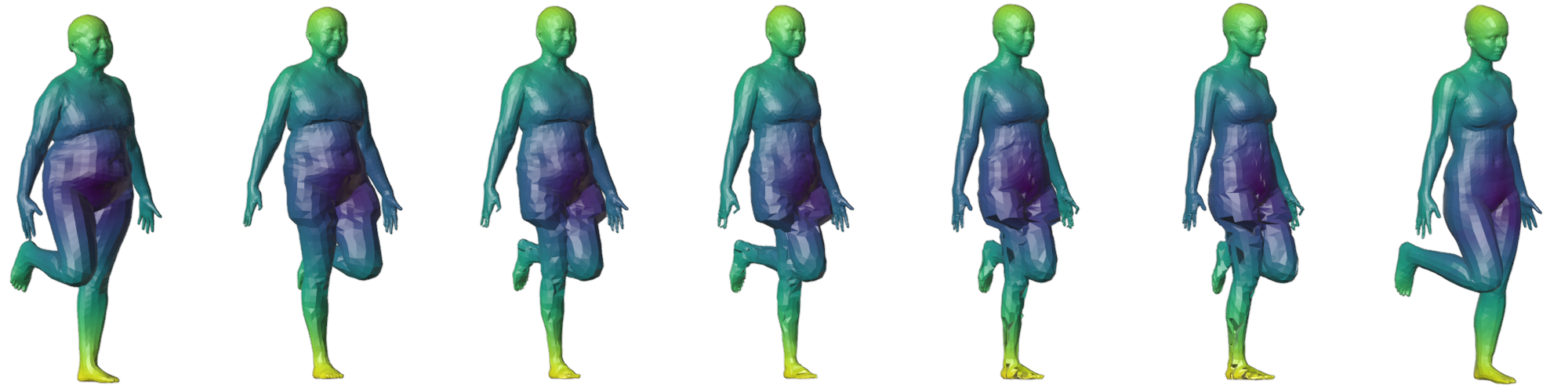}
    \\
    \setlength{\labwidth}{50pt}
    \parbox{\labwidth}{\centering 
    \small{\hspace{-10pt}diam: \, 1.62 \\ \hspace{-10pt}PCA: \, ---}
    }
    \parbox{\labwidth}{\centering 
    \small{1.62 \\ 0.06}
    }
    \parbox{\labwidth}{\centering 
    \small{1.63 \\ 0.07}
    }
    \parbox{\labwidth}{\centering 
    \small{1.63 \\ 0.08}
    }
    \parbox{\labwidth}{\centering 
    \small{1.64 \\ 0.07}
    }
    \parbox{\labwidth}{\centering 
    \small{1.64 \\ 0.06} 
    }
    \parbox{\labwidth}{\centering 
    \small{1.65 \\ ---}
    }
    \\
    \includegraphics[
    width=0.9\linewidth]{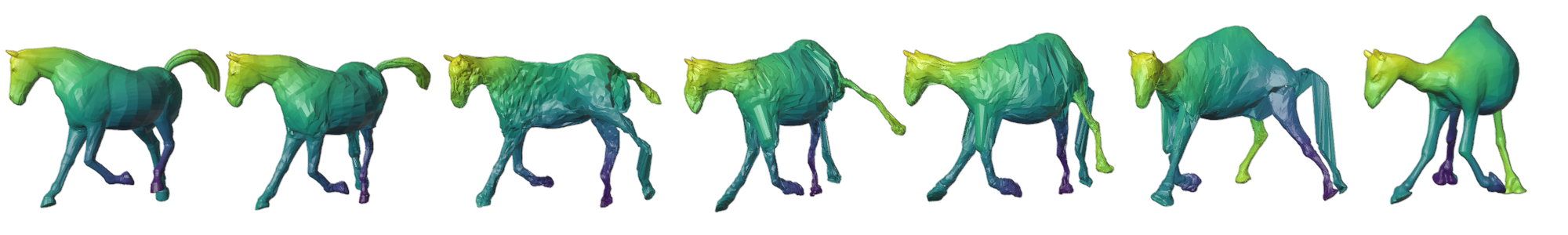}
    \\
    \setlength{\labwidth}{50pt}
    \parbox{\labwidth}{\centering 
    \small{\hspace{-10pt}diam: \, 1.10 \\ \hspace{-10pt}PCA: \, ---}
    }
    \parbox{\labwidth}{\centering 
    \small{1.07 \\ 0.11}
    }
    \parbox{\labwidth}{\centering 
    \small{1.09 \\ 0.15}
    }
    \parbox{\labwidth}{\centering 
    \small{1.11 \\ 0.18}
    }
    \parbox{\labwidth}{\centering 
    \small{1.13 \\ 0.18}
    }
    \parbox{\labwidth}{\centering 
    \small{1.15 \\ 0.14} 
    }
    \parbox{\labwidth}{\centering 
    \small{1.17 \\ ---}
    }
    \\[10pt]
    \caption{Gromov--Wasserstein
    interpolations $(\YY_\rho)_{\rho \in \Delta_{1}}$ between two 3d shapes for a total of four input pairs $(\XX_1,\XX_2)$.
    From left to right, each row shows
    $\XX_1,
    \YY_{(\nicefrac{5}{6},\nicefrac{1}{6})},
    \YY_{(\nicefrac{4}{6},\nicefrac{2}{6})},
    \YY_{(\nicefrac{3}{6},\nicefrac{3}{6})},
    \YY_{(\nicefrac{2}{6},\nicefrac{4}{6})},
    \YY_{(\nicefrac{1}{6},\nicefrac{5}{6})},
    \XX_2$.
    The colouring indicates the optimal GW plan
    between the inputs and interpolants.
    The Euclidean diameters 
    (before applying the PCA)
    and PCA residuals are shown below each input/barycenter.
    The GW distances are (top to bottom):
    0.07,\, 0.03,\, 0.06,\, 0.15.}
    \label{fig:2-interpolations}
\end{figure}

For $N=2$, \cref{fig:2-interpolations} 
shows the obtained interpolations 
for $(\rho_1,\rho_2) 
= (\nicefrac{5}{6},\nicefrac{1}{6}),
\dotsc,
(\nicefrac{1}{6},\nicefrac{5}{6})$
between various choices of meshes from both mentioned datasets.
The colouring illustrates the transport between the inputs
as well as between input and barycenter.
We remark that the obtained interpolants do not 
follow any particular fixed alignment
but have been rotated manually to align with the inputs
for the purpose of comparison.
In relation to the Euclidean diameters,
the mean PCA errors, 
i.e.\ the mean distance of points before and after the PCA,
are reasonably small.
The results indicate that
this feature is especially pronounced,
if the GW distance
between the associated inputs
is small as well.
Judging from a qualitative viewpoint, 
up to some small potential error,
the obtained transport plans associated 
to $\XX_1,\XX_2$ of rows one and two 
are most likely global minimizers of $\Fgw^{\XX_1,\XX_2}$.
The optimality of the plan results in a 
very good shape interpolation between the inputs.
In row three we see the results in the case where the
obtained transport plan is 
most likely
only a local minimum.
As can be seen from the colours, 
the plan matches the left feet and right feet
of the subjects correctly, but matches the left hand of
the first subject with the right hand of the second 
and vice versa, which is evidently not a global minimum.
This non-optimal matching is well reflected 
by the interpolation in which
the legs seem to change places.
Finally, the last row shows the results 
for a horse and a camel.
It should be noted that the camel has, 
in contrast to the horse,
a very short tail.
Geodesically speaking, 
the inputs are thus very different.
Indeed, as the colouring shows,
the obtained transport plan matches the horse's tail
to one of the camels legs 
and partially merges two of the horse's legs to one.
In the interpolation,
the tail moves down the body 
to eventually become the leg of the camel.

The interpolation between $N=4$ inputs 
are shown in \cref{fig:4-interpolations}.
Different to the FAUST interpolations above,
we picked four distinct subjects
with distinct poses.
Therefore, the challenge here
is not only the interpolation between
different body shapes
but also between the poses.
Judging from a qualitative viewpoint,
our proposed method seems to overcome this
suitably well.
We want to emphasize again,
that the entire interpolation
is obtained by 3 iterations of \cref{alg:1}.
The quality of it suggests that
the obtained melting $\mu$
seems to be at least near-optimal 
for $\MGW_\rho(\XX_1,\dotsc,\XX_4)$
independently of $\rho \in \Delta_3$.
Overall,
the surfaces of the interpolants
seem to be less smooth
than those of the inputs.
Here, 
the mean over all 
mean pca errors
of all 21 barycenters 
is 0.11 
and the mean Euclidean diameter
of the 4 inputs is 1.92.
which indicates that the PCA
achieves a reasonable approximation.
\begin{figure}
    \centering
    \includegraphics[width=220pt]{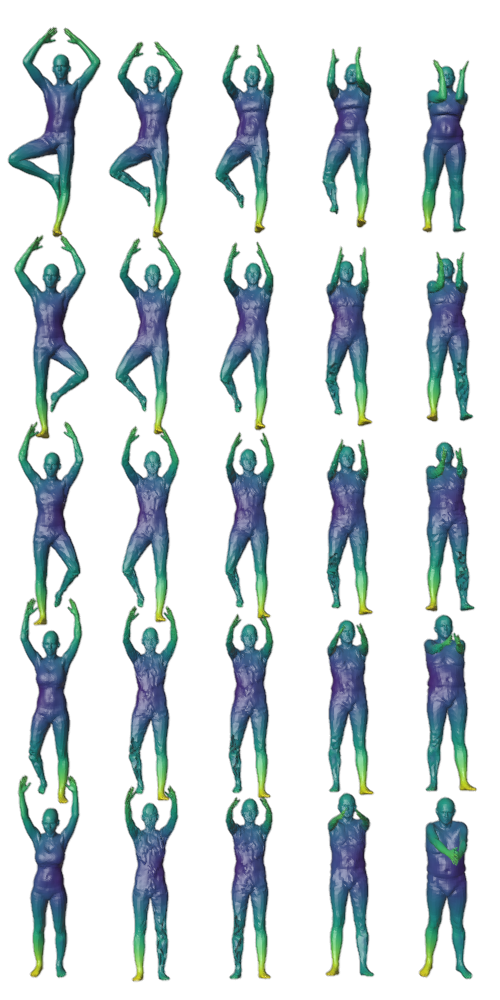}
    \caption{A GW interpolation 
    $(\YY_{\rho})_{\rho \in \Delta_3}$
    between four input shapes $\XX_1,\XX_2,\XX_3,\XX_4$ shown in the corners.}
    \label{fig:4-interpolations}
\end{figure}

\subsection{Approximation of Pairwise Distances}\label{subsec:lgw}

According to \cref{subsec:tangent_distance_approximation}, 
\cref{alg:1} provides us with an approximation of
the pairwise linear GW distances of the inputs
at a reference point.
In this example, we show how these distances
can be leveraged to classify shapes.
As above, we focus on the training set of the FAUST dataset
and the meshes from the deformation dataset.
The former consists of 10 human subjects in 10 poses each,
totaling to 100 meshes,
whereas the latter consists in particular of the classes
camel, cat, elephant, face, head, horse, lion,
each of which is given in 10 or 11 poses.
We reduce all meshes down to 500 nodes using the same
quadric mesh decimation and extract the gm-spaces 
as in \cref{subsec:interpolation}.
We expect that identifying the classes from the deformation dataset with GW poses an easier task 
than identifying the subjects from the FAUST training set.
This is due to the fact that the geometry 
between meshes from distinct classes 
of the deformation dataset varies much more 
than the geometry between distinct human subjects.

In the case of the deformation dataset, 
we extract $N = 73$ gm-spaces 
$\XX_1,\dotsc,\XX_N$.
All gm-spaces are supported on $500$ points
and are endowed with their pairwise Djikstra distances.
As the meshes of distinct classes vary greatly in diameter,
we rescale the gm-spaces so that the 
maximal Djikstra distance of each is one.
The classifcation task would be too simple otherwise.
We proceed to compute all pairwise GW distances 
of the entire set 
with block-coordinate descent
which takes approximately 43 minutes
on our machine.
As previously mentioned, 
the conditional gradient might
only recover local minima.
We use the following method to mitigate this problem.
Since $\GW$ is a metric, 
we check if the triangle inequality holds for all inputs. 
Let $\pi_{i,j}$ be the obtained transport plan between
$\XX_i$ and $\XX_j$ for all $i,j \in \{1,\dotsc,N\}$
and set \smash{$\GW_{i,j} = \Fgw^{\XX_i,\XX_j}(\pi_{i,j})$}.
Whenever there exists $k \in \{1,\dotsc,N\}$ 
so that
$
\GW_{i,k} + \GW_{k,j} < \GW_{i,j}
$
we know that $\pi_{i,j} \not\in \Pio(\XX_i,\XX_j)$.
In this case, 
we restart the block-coordinate descent
for inputs $\XX_i$, $\XX_j$ with initial value
$(P_{X_i \times X_j})_\# \gamma$, 
where $\gamma \in \Gamma_{\XX_k}(\pi_{i,k},\pi_{k,j})$
is the north-west-corner gluing 
between $\pi_{i,k}$ and $\pi_{k,j}$ along $\XX_k$.
The obtained pairwise distances are shown on the left-hand side of \cref{fig:gw_and_conf_DF}. 
As can be seen, 
the matrix of the pairwise distances 
readily identifies the classes.
We verify this
by a simple nearest-neighbour procedure.
Here we run $10\,000$ iterations,
in each of which we pick a random representative 
of each class.
Afterwards, all shapes are classified 
using nearest neighbour classification with respect
to the selected representatives.
The confusion matrix is then estimated by considering 
the number of times a shape from class $n$ 
has been classified as class $m$, 
for $n,m \in 
\{\text{camel}, \text{cat}, \text{elephant}, 
\text{face}, \text{head}, \text{horse}, \text{lion}\}$.
The result is then normalized 
by the number of iterations 
and by the class-size.
The resulting matrix is shown on
the right-hand side of \cref{fig:gw_and_conf_DF}.

\begin{figure}[t]
    \centering
    \includegraphics[width = 100pt]{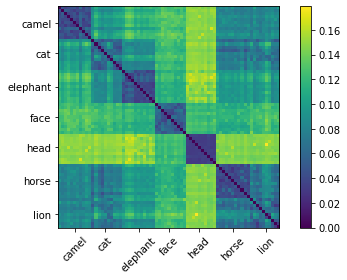}
    \includegraphics[width = 100pt]{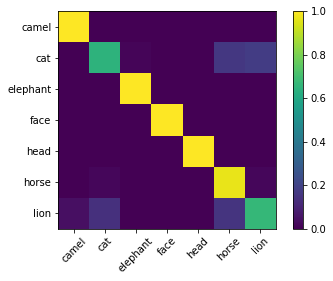}
    \caption{Pairwise GW distances (left) 
    and confusion matrix (right)
    of the deformation dataset.}
    \label{fig:gw_and_conf_DF}
\end{figure}

To quantify how well the linear approximation
compares against the obtained $\GW$ distances,
we run a 10-fold cross validation paired
with a support vector machine 
to evaluate the potency for classification.
More precisely, 
we split the set of $N$ shapes into 10 subsets
and perform 10 iterations in each of which 
9 subsets are considered as a training set,
whereas the remaining one is taken as a test set.
In each iteration of the cross validation 
we do the following.
Firstly, we run $5$ iterations of $\cref{alg:1}$,
initialized by setting $\YY$ 
as a random element of the training set.
In the gluing-melting step, 
we use the north-west corner rule.
For $k=1,\dotsc,5$, let $\mu_k$
be the obtained melting after the $k$-th iteration.
As elaborated in \cref{subsec:tangent_distance_approximation},
an approximation of the linear GW distance
between $\XX_i,\XX_j$ at
the reference $\YY^{(k-1)}$ is then obtained by setting
\[
\LGW^{(k)}_{i,j}
\coloneqq 
\Fgw^{\XX_i,\XX_j}( (P_{X_i \times X_j})_\# \mu_k).
\]
In the following, for each $k =1,\dotsc,5$,
we compare $\LGW^{(k)}$ with $\GW$, 
within a larger machine-learning pipeline.
To this end, we train two support vector machines 
with the kernels 
\smash{$\exp(-10\GW_{i,j})_{i,j}$} 
and \smash{$\exp(-10\LGW^{(k)}_{i,j})_{i,j}$} 
on the training set and evaluate on the test set.
Additionally, 
we compute the mean relative error (MRE)
and the Pearson correlation coefficient (PCC),
of $(\GW_{i,j})_{i,j}$ with respect to \smash{$(\LGW^{(k)}_{i,j})_{i,j}$}, $k =1,\dotsc,5$, where the mean is taken over all iterations of the cross validation.
We remark that for the MRE, 
the $0$ values of the matrices are not considered.
The results are shown in \cref{tab:LGW_DF}.
\begin{table}
\caption{
Results of the 10-fold cross validation for the deformation dataset.
The MRE and PCC are the means for the respective
iterations over all iterations.
}
\label{tab:LGW_DF}
\centering
\begin{tabular}{lcccc}
    \toprule
     & Time (min) & Acc (\%) & MRE & PCC
    \\ 
    \midrule
    $\LGW^{(1)}$      
    & 1 & $83 \pm 15$ & $0.37 \pm 0.14$ & $0.73 \pm 0.16$
    \\
    $\LGW^{(2)}$      
    & 2 & $90 \pm 13$ & $0.25 \pm 0.08$ & $0.87 \pm 0.10$
    \\
    $\LGW^{(3)}$      
    & 3 & $95 \pm \,\,\,\,7$ & $0.18 \pm 0.03$ & $0.94 \pm 0.01$
    \\
    $\LGW^{(4)}$      
    & 4 & $96 \pm \,\,\,\,6$ & $0.17 \pm 0.02$ & $0.95 \pm 0.02$
    \\
    \midrule
    $\LGW^{(5)}$      
    & 5 & $97 \pm \,\,\,\,6$ & $0.15 \pm 0.01$ & $0.96 \pm 0.01$
    \\
    $\GW$      
    & 43 & $97 \pm \,\,\,\,6$ & --- & ---
    \\
    \bottomrule
\end{tabular}
\end{table}
As we expected, $\GW$ is well suited 
for classifying the given shapes.
The cross-validation shows that $\GW$ achieves
an almost perfect classification.
The linear approximation
$\LGW^{(k)}$
achieves monotonous improvement in terms of
accuracy, MRE and PCC. 
Furthermore, 
$\LGW^{(5)}$ achieves the same accuracy 
as the $\GW$ in terms of the cross validation process.

We proceed analogously for $N=100$ training meshes
of the FAUST dataset.
In this case, we have exactly 10 classes 
with 10 meshes per class.
The pairwise distances together with the
estimated confusion matrix are shown in \cref{fig:gw_and_conf_DF}.
\begin{figure}[t]
    \centering
        \includegraphics[height=80pt]{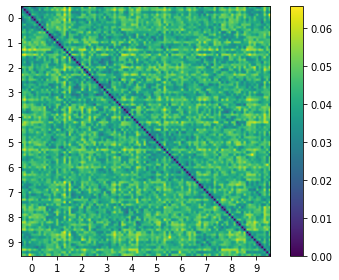}
        \includegraphics[height=80pt]{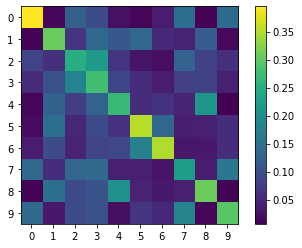}
    \caption{Pairwise GW distances (left) 
    and confusion matrix (right)
    of the Faust dataset.}
    \label{fig:gw_and_conf_F}
\end{figure}
The figure shows that, in terms of the GW distance, 
the differences between the classes are less pronounced.
The results of the analogous 10-fold cross validation
are presented in \cref{tab:LGW_F}.
\begin{table}
\caption{
Results of the 10-fold cross validation for the FAUST dataset.
The MRE and PCC are the means for the respective
iterations over all iterations.
}
\label{tab:LGW_F}
\centering
\begin{tabular}{lcccc}
    \toprule
     & Time (min) & Acc (\%) & MRE & PCC
    \\ 
    \midrule
    $\LGW^{(1)}$      
    & 1 & $31 \pm 8$ & $0.19 \pm 0.07$ & $0.61 \pm 0.16$
    \\
    $\LGW^{(2)}$      
    & 2 & $41 \pm 14$ & $0.15 \pm 0.01$ & $0.64 \pm 0.10$
    \\
    $\LGW^{(3)}$      
    & 3 & $34 \pm 13$ & $0.15 \pm 0.01$ & $0.66 \pm 0.07$
    \\
    $\LGW^{(4)}$      
    & 4 & $44 \pm 12$ & $0.15 \pm 0.02$ & $0.63 \pm 0.11$
    \\
    \midrule
    $\LGW^{(5)}$      
    & 5 & $39 \pm 16$ & $0.15 \pm 0.01$ & $0.63 \pm 0.10$
    \\
    $\GW$      
    & 52 & $43 \pm 13$ & --- & ---
    \\
    \bottomrule
\end{tabular}
\end{table}
As we expected, 
$\GW$ does not provide a very accurate classification
of the dataset.
In terms of the MRE and PCC, the linear approximation 
seems to converge faster than for the previous dataset.
We postulate that this is due to the fact, 
that the dataset already consists of objects which lie all relatively close with respect to the GW distance.
Thus the initial choice $\YY$ 
(which is one of the input gm-spaces)
already achieves a small barycenter loss.
The accuracy based on the $\LGW$
seems to be less stable through the iterations
but approximates the one based on $\GW$ suitably well.

\subsection{Multi-graph matching of Protein networks}\label{subsec:PPI}

In this example, 
we explore our proposed methods potential 
for the task of multi-graph matching.
We consider a dataset of $N=6$ yeast networks 
which has been curated by Collins et al. \cite{PPI}%
\footnote{The dataset can be downloaded from
\url{https://www3.nd.edu/~cone/MAGNA++/}.}.
The dataset has been used as a benchmark 
in e.g.\ \cite{magna1,magna2}.
In addition,
the dataset has been considered in \cite{xu2019scalable} 
to gauge the effectiveness of 
a $\GW$-based method to 
achieve a multi-graph matching.
Although it seems that the authors utilize
a GW barycenter for their proposed
divide-and-conquer approach,
the publicized source code
reveals that the
reported results are obtained
by a greedy pairwise method.
In the following,
we want to show that our tangential barycenter
may be leveraged to obtain a viable
multi-graph matching.
We enumerate the yeast networks by $i =1,\dotsc,6$.
The first network ($i=1$) comes with
8323 high confidence 
protein-protein interactions (PPI).
In addition to these high confidence PPIs, 
the other five networks ($i=2,\dotsc,6$)
also exhibit $p$\% (directed) low-confidence PPIs
for $p \in \{5,10,15,20,25\}$.
For each network their PPIs 
are stored as a matrix 
$M^{(i)} \in \{0,1\}^{1004 \times 1004}$,
where
\[
M_{k,l}^{(i)} = 
\begin{cases}
0   &\text{protein $k$ interacts with protein $l$.}\\
1   &\text{protein $k$ does not interact with protein $l$} 
\end{cases}.
\]
Although PPIs are naturally symmetric,
for reasons unaware to us,
the stored matrices $M^{(i)}$ are not.
Below we will resort to simple symmetrization
to bypass this circumstance.
Finally, on every network a
probability measure $\xi_i$ on the set of nodes is available, 
$i = 1,\dotsc,6$.
The dataset comes with 
the ground truth one-to-one correspondences
between all networks.

We seek to apply our proposed method 
to achieve a multi-alignment of the given protein networks.
To this end, let 
$X_1 = \dotsc = X_6 = \{1,\dotsc,1004\}$,
$g_i = \frac{1}{2} (M^{(i)} + (M^{(i)})^\tT)$ so that
each network gives rise to a gm-space
$\XX_i \coloneqq (X_i,g_i,\xi_i)$.
We proceed to apply \cref{alg:1} four times on 
$\XX_1,\dotsc,\XX_n$, 
one time for each $n \in \{3,\dotsc,6\}$
with initialization $\YY = \XX_1$
and under the use of a proximal gradient descent
to solve the inner GW computations.
As we are interested in one-to-one matchings 
of the respective inputs
the measures are constructed according to 
the maximum rule without replacement,
see \cref{sec:in_practice}.
In this way, after each iteration of \cref{alg:1},
we obtain a gm-space of the form
$\YY_n \coloneqq (\bigtimes_{i=1}^n X_i, \frac{1}{n} \sum_{i=1}^n  g_i, \mu_n)$,
where $\lvert \supp \mu_n \rvert = 1004$.
We iterate \cref{alg:1} until the barycentric loss 
$\Fgwb^{\XX_1,\dotsc,\XX_n}$ increases
which results in 4,2,2,1 iterations for
$n = 3,4,5,6$, respectively.
The multi-marginal matching between the $n$ inputs
is then encoded in the support points of the obtained
multi-marginal plan $\mu_n$, i.e.
\[
(x_1,\dotsc,x_n) \in \bigtimes_{i=1}^n X_i
\quad \text{ are matched if } \quad
(x_1,\dotsc,x_n) \in \supp(\mu_n).
\]
Inspired by \cite{xu2019scalable}, 
we employ the following two evaluation
measures for node correctness
\begin{align*}
\text{NC@1} &\coloneqq 
\frac
{\lvert 
\{x \in \supp(\mu_n) : 
\text{at least one pair $(x_k,x_l)$ for $k \neq l$ is matched correctly}\} 
\rvert}
{1004}
\\
\text{NC@A} &\coloneqq 
\frac
{\lvert 
\{x \in \supp(\mu_n) : 
\text{ all $x_1,\dotsc,x_n$ are matched correctly}
\} 
\rvert}
{1004}.
\end{align*}
The results as well as a comparison 
with those of \cite{xu2019scalable}
as they are reported in the reference
are provided in \cref{tab:multi-matching}.
It should be noted that
the authors of \cite{xu2019scalable}
are using the non-symmetric
matrices $M^{(i)}$ above
as gauge functions for their methods.
For a fair comparison against our method (GWTB),
we repeat their experiments after replacing
the adjacency matrices with the gauge functions $g_i$
defined above which results are presented under GWL (sym). 
After symmetrization,
the methods S-GWL and GWL
yielded the same results.
As we can see, 
our method performs (slightly) better
than GWL (sym).

\begin{table}
\caption{Comparison of GWTB multi-graph matching with
the one proposed in \cite{xu2019scalable}
for yeast network dataset.}
\centering
\setlength{\tabcolsep}{5pt}
\begin{tabular}{ccccccccc
} 
\toprule
\multirow{2}{*}{Method} &
\multicolumn{2}{c}{3 graphs} &
\multicolumn{2}{c}{4 graphs} &
\multicolumn{2}{c}{5 graphs} &
\multicolumn{2}{c}{6 graphs}\\ \cline{2-9}
&
NC@1 &NC@all &
NC@1 &NC@all &
NC@1 &NC@all &
NC@1 &NC@all \\ \hline
GWL
&63.84 &46.22
&68.73 &39.14
&71.61 &31.57
&76.49 &28.39\\
S-GWL
&60.06 &43.33
&68.53 &38.45
&73.21 &33.27
&76.99 &29.68\\
GWL (sym)
&89.34 &69.42
&94.12 &56.27
&97.01 &49.40
&98.01 &\textbf{45.52}\\
GWTB (ours)
&\textbf{93.32} &\textbf{69.92}
&\textbf{96.81} &\textbf{61.35}
&\textbf{97.21} &\textbf{52.09}
&\textbf{98.90} &\textbf{45.52}\\
\bottomrule
\end{tabular}\label{tab:multi-matching}
\end{table}

\section*{Conclusions}

    In this paper,
    we show existence and characterize 
    tangential GW barycenters,
    which gives rise to a novel method 
    for approximating
    GW barycenters.
    We prove that the latter 
    monotonously decreases 
    the barycenter functional.
    We give numerical evidence 
    that our proposed method
    can be used to obtain entire
    GW interpolations
    between multiple 3d shapes,
    for approximation of pairwise GW distances,
    as well as for multi-graph matching tasks.
    Another line of work,
    namely fused GW,
    aims at finding transport plans
    based on both the pairwise preservation 
    of internal gauges as well as the preservation
    of labels in an additional label space.
    A generalization of our work
    to the fused case would require
    a careful study of the geometric structure
    of the labelled gm-spaces endowed with the fused GW distance which we leave as future work.

\section*{Funding}
This work is supported in part by funds from the German Research Foundation (DFG) within the RTG 2433 DAEDALUS.

\appendix

\section{Proof of \texorpdfstring{\cref{thm:gen-bary}}{Theorem 5.1}}
\label{sec:proof-char-bary}
    The following is a slight generalization
    of the proof of
    \cite[Thm~5.1]{BBS2022multi}.
    First, notice that for 
    for real numbers $a_1,\dotsc,a_N \in \R$,
    it holds
    \begin{equation*}
    \sum_{i,j=1}^N \rho_i \rho_j 
    \lvert a_i - a_j \rvert^2
    = \sum_{i=1}^N \rho_i a_i^2
    - \sum_{i,j=1}^N \rho_i \rho_j a_i a_j
    = \sum_{i=1}^N \rho_i \lvert a_i - \sum_{j=1}^N \rho_j a_j \rvert.
    \end{equation*}
    Thus
    \begin{equation}\label{eq:aa2}
    \min_{b \in \R} \sum_{i=1}^N \rho_i \lvert a_i - b \rvert^2
    = 
    \sum_{i=1}^N \rho_i 
    \lvert a_i -  \sum_{j=1}^N \rho_j a_j \rvert^2
    = 
    \sum_{i,j=1}^N \rho_i \rho_j 
    \lvert a_i - a_j \rvert^2.
    \end{equation}
    Hence, the integrand of 
    $\Fmgw^{\XX_1,\dots,\XX_N} (\pi)$,
    $\pi \in \Pi(\XX_1,\dotsc,\XX_N)$,
    becomes
    \begin{align*}
    \sum_{i,j=1}^N \rho_i\rho_j \; 
    \lvert g_i(x_i,x_i') - g_j(x_j,x'_j) \rvert^2
    &=
    \sum_{i=1}^N \rho_i \Bigl\lvert g_i(x_i,x_i') - \sum_{j=1}^N \rho_j g_j(x_j,x_j') \Bigr\rvert
    \\
    &= 
    \sum_{i=1}^N \rho_i \lvert g_i(x_i,x_i') - m_\rho(x_\times,x_\times') \rvert.
    \end{align*}
    Now denote 
    $\Xf_i 
    = \llbracket \XX_i \rrbracket 
    = \llbracket (X_i,g_i,\xi_i) \rrbracket$
    and let 
    $\Yf 
    = \llbracket \YY \rrbracket
    = \llbracket (Y,h,\upsilon) \rrbracket$
    be arbitrary.
    For $i =1,\dotsc,N$ consider 
    $\pi_i \in \Pio(\XX_i,\YY)$ 
    and let
    $\gamma 
    \in \Gamma_\YY(\pi_1,\dotsc,\pi_N)$.
    Pointwisely applying \cref{eq:aa2}
    with $a_i = g_i(x_i,x_i')$ 
    yields
    \begin{align}\label{eq:Y_lower_bound_MGW}
        \Fgwb^{\XX_1,\dotsc,\XX_N}(\Yf) 
        = &\sum_{i=1}^N \rho_i \GW(\XX_i,\YY)
        = \iint_{(X_\times \times Y)^2}
        \sum_{i=1}^N \rho_i \lvert g_i - h \rvert \dx \gamma \dx \gamma \nonumber
        \\
        \geq 
        &\iint_{(X_\times \times Y)^2}
        \sum_{i,j=1}^N \rho_i \rho_j \lvert g_i - g_j \rvert \dx \gamma \dx \gamma
        \nonumber
        \\
        =
        &\iint_{X_\times^2}
        \sum_{i,j=1}^N \rho_i \rho_j \lvert g_i - g_j \rvert \dx (P_{X_\times})_\# \gamma \dx (P_{X_\times})_\# \gamma
        \geq \MGW_\rho(\XX_1,\dotsc,\XX_N),
    \end{align}
    where the last estimate follows by
    $(P_{X_\times})_\# \gamma 
    \in \Pi(\XX_1,\dotsc,\XX_N)$.
    We show that this lower bound is attained
    for $\hat{\Yf }
    = \llbracket \hat{\YY} \rrbracket 
    = \llbracket(X_\times, m_\rho,\hat{\pi}) \rrbracket$,
    where $\hat{\pi} \in \Pio(\XX_1,\dotsc,\XX_N)$
    is arbitrary.
    Indeed, by defining 
    $\hat{\pi}_i 
    \coloneqq (P_{X_i},P_{X_\times})_\# \hat{\pi}$
    and again using \cref{eq:aa2},
    it holds
    \begin{align*}
        &\MGW_\rho(\XX_1,\dotsc,\XX_N)
        \\
        =
        &\Fmgw^{\XX_1,\dotsc,\XX_N}(\hat{\pi})
        =
        \iint_{X_\times^2}
        \sum_{i,j=1}^N \rho_i \rho_j \lvert g_i - g_j \rvert \dx \hat{\pi} \dx \hat{\pi}
        =
        \iint_{X_\times^2}
        \sum_{i=1}^N \rho_i 
        \lvert g_i - m_\rho \rvert 
        \dx \hat{\pi} \dx \hat{\pi}
        \\
        =
        &\sum_{i=1}^N \rho_i 
        \iint_{(X_i \times X_\times)^2}
        \lvert g_i(x_i,x_i') - m_\rho(y,y') \rvert 
        \dx \hat{\pi}_i(x_i,y) \dx \hat{\pi}_i(x_i',y')
        \geq \sum_{i=1}^N \rho_i \GW(\XX_i,\hat{\YY})
        \\
        &= \Fgwb^{\XX_1,\dotsc,\XX_N}(\hat{\Yf}),
    \end{align*}
    where the last estimate is due to
    $\hat{\pi}_i \in \Pi(\XX_i,\hat{\YY})$.
    By the first part, 
    $\hat{\Yf}$ is a solution to
    $\GWB_\rho(\Xf_1,\dotsc,\Xf_N)$.
    
    Now, let 
    $\tilde{\Yf}
    = \llbracket \tilde{\YY} \rrbracket
    = \llbracket (\tilde{Y},\tilde{h},\tilde{\upsilon}) \rrbracket$
    be any minimizer of 
    $\GWB_\rho(\Xf_1,\dotsc,\Xf_N)$.
    We construct 
    $\hat{\pi} \in \Pio(\XX_1,\dotsc,\XX_N)$
    so that
    $\tilde{\Yf} 
    = \llbracket 
    (X_\times,m_\rho,\hat{\pi}) 
    \rrbracket$.
    Let 
    $\tilde{\gamma} 
    \in \Gamma_{\tilde{\YY}}(\tilde{\pi}_1,\dotsc,\tilde{\pi}_N)$, 
    where 
    $\tilde{\pi}_i \in \Pio(\XX_i,\tilde{\YY})$.
    By repeating the steps of \cref{eq:Y_lower_bound_MGW}
    for $\Yf = \tilde{\Yf}$
    and using that
    $\tilde{\Yf}$ is a solution of
    $\GWB_\rho(\Xf_1,\dotsc,\Xf_N)$,
    we firstly obtain 
    $\tilde{\pi} 
    \coloneqq (P_{X_\times})_\# \tilde{\gamma}
    \in \Pio^\rho(\XX_1,\dotsc,\XX_N)$
    and secondly
    \begin{equation}\label{eq:tildegamma_a_e}\tilde{h}(\cdot_1,\cdot_3) = m_\rho(\cdot_2,\cdot_4)
    \quad (\tilde{\gamma} \otimes \tilde{\gamma})(\cdot_1,\cdot_2,\cdot_3,\cdot_4)\text{ - a.s.},
    \end{equation}
    where we also used \cref{eq:aa2} 
    to receive the latter.
    Set $\hat{\YY} = (X_\times,m_\rho,\hat{\pi})$.
    By construction we have
    $\tilde{\gamma} \in \Pi(\tilde{\YY},\hat{\YY})$.
    Using this together with \cref{eq:tildegamma_a_e}, we obtain
    \[
    \GW(\tilde{\YY},\hat{\YY})^2
    \leq 
    \iint_{(\tilde{Y} \times X_\times)^2}
    \underbrace{\lvert \tilde{h} - m_\rho \rvert^2}_{=0 \text{ a.e.}}
    \dx \tilde{\gamma} \dx \tilde{\gamma}
    = 0,
    \]
    which concludes the proof.

\newpage
\section{Proof of \texorpdfstring{\cref{thm:gaussian}}{Theorem 5.2}}
\label{sec:proof-gaussian}

Let $\rho \in \Delta_{N-1}$ be arbitrary
and set 
$\hat{\pi} \coloneqq (T_1,\dotsc,T_N)_\# \xi_1$.
Note that
$\xi_1
= 
(T_1)_\# \xi_1$.
Therefore, 
without loss of generality,
we set $T_1 = \id_{\R^{d_1}}$.
In the following, 
we rely on
\cite[Prop~4.1]{delon2022gromov}
which states
\[
(P_{X_1 \times X_i})_\# \hat{\pi} 
= (\id_{\R^{d_1}},T_i)_\# \xi_1 \in \Pio(\XX_1,\XX_i).
\]
Consequently,
$\hat{\pi} \in \Pi(\XX_1,\dotsc,\XX_N)$.
Next, we show
$(P_{X_i \times X_j})_\# \pi \in \Pio(\XX_i,\XX_j)$
for all $i,j = 1,\dotsc,N$,
which yields
$\pi \in \Pio^\rho(\XX_1,\dotsc,\XX_N)$
due to
\begin{align*}
\frac{1}{2} \sum_{i,j=1}^N \rho_i \rho_j \GW(\XX_i,\XX_j)
\leq 
&\MGW_\rho(\XX_1,\dotsc,\XX_N)
\leq 
\Fmgw^{\XX_1,\dotsc,\XX_N}(\hat{\pi})
\\[-10pt]
= 
&\frac{1}{2} \sum_{i,j=1}^N \rho_i \rho_j
\Fgw^{\XX_i, \XX_j}((P_{X_i \times X_j})_\# \hat{\pi})
= 
\frac{1}{2} \sum_{i,j=1}^N \rho_i \rho_j \GW(\XX_i,\XX_j).
\end{align*}
Let $i \leq j$ and set
\begin{equation*}
S:\R^{d_i} \to \R^{d_j}, 
\quad
S(x) = P_j B P_i^\tT x, 
\quad B =
\Bigl(
    \tilde{I}_{d_j} \tilde{I}_{d_i}^{(d_j)} 
    D_j^{\nicefrac{1}{2}} 
    \Bigl(D_i^{(d_j)}\Bigr)^{-\nicefrac{1}{2}} 
    \Bigm| 0_{d_j, d_i - d_j}
\Bigr)
\in \R^{d_j \times d_i}.
\end{equation*}
We have
\begin{align*}
B A_i 
=
&\Bigl(
    \tilde{I}_{d_j} \tilde{I}_{d_i}^{(d_j)} 
    D_j^{\nicefrac{1}{2}} 
    \Bigl(D_i^{(d_j)}\Bigr)^{-\nicefrac{1}{2}} 
    \Bigm| 0_{d_j, d_i - d_j}
\Bigr)
\Bigl(
    \tilde{I}_{d_i} 
    D_i^{\nicefrac{1}{2}} 
    \Bigl(D_1^{(d_i)}\Bigr)^{-\nicefrac{1}{2}} 
    \Bigm| 0_{d_i, d_1 - d_i}
\Bigr)
\\
=
&\Bigl(
    \tilde{I}_{d_j} \tilde{I}_{d_i}^{(d_j)} 
    D_j^{\nicefrac{1}{2}} 
    \Bigl(D_i^{(d_j)}\Bigr)^{-\nicefrac{1}{2}} 
    \tilde{I}_{d_i}^{(d_j)}
    \bigl(D_i^{\nicefrac{1}{2}}\bigr)^{(d_j)} 
    \Bigl(\Bigl(D_1^{(d_i)}\Bigr)^{-\nicefrac{1}{2}}\Bigr)^{(d_j)}
    \Bigm| 0_{d_j, d_1 - d_j}
\Bigr).
\end{align*}
As all matrices in the left block are diagonal, 
they commute.
Applying this together with
\smash{$(D_i^{\nicefrac{1}{2}})^{(d_j)} 
= (D_i^{(d_j)})^{\nicefrac{1}{2}}$}
and
\smash{
$((D_1^{(d_i)})^{-\nicefrac{1}{2}})^{(d_j)}
= ((D_1^{(d_j)})^{-\nicefrac{1}{2}})$
},
we obtain
\begin{equation*}
    B A_i
    =
    \Bigl(
    \tilde{I}_{d_j}
    D_j^{\nicefrac{1}{2}} 
    \Bigl(D_1^{(d_j)}\Bigr)^{-\nicefrac{1}{2}}
    \Bigm| 0_{d_j, d_1 - d_j}
\Bigr)
= A_j.
\end{equation*}
Hence,
$S \circ T_i 
= P_j B P_i^\tT  P_i A_i P_1^\tT
= P_j B A_i P_1^\tT
= P_j A_j P_1^\tT
= T_j$.
This shows the first part due to
\[
(P_{X_i \times X_j})_\# \hat{\pi} 
= (T_i,T_j)_\# \xi_1
= (T_i,S T_i)_\# \xi_1
= (\id_{\R^{d_i}},S)_\# \xi_i
\in \Pio(\XX_i,\XX_j),
\]
where the latter inclusion is again provided
by \cite[Prop.~4.1]{delon2022gromov}.

We turn our attention 
to the barycenter statement.
As $\GWB_{\rho}(\XX_1,\dotsc,\XX_N)$ is independent of 
orthogonal transformations of the inputs,
we assume without loss of generality
that $\Sigma_{i} = D_i$ which gives $T_i = A_i$.
Let $d_\oplus \coloneqq \sum_{i=1}^N d_i$.
Combining the first part with \cref{thm:gen-bary} 
yields that the gm-space
\[
\bigl(\R^{d_\oplus}, 
\sum\nolimits_{i=1}^N \rho_i \langle \cdot, \cdot \rangle_{d_i}, 
\hat{\pi}\bigr)
=
\bigl(\R^{d_\oplus}, 
\sum\nolimits_{i=1}^N \rho_i \langle \cdot, \cdot \rangle_{d_i}, 
(A_1,\dotsc,A_N)_\# \xi_1\bigr),
\]
is a solution of 
$\GWB_{\rho}(\XX_1,\dotsc,\XX_N)$
for all $\rho \in \Delta_{N-1}$.  
For $x,x' \in \R^{d_1}$, it holds
\[
\sum_{i=1}^N \rho_i \langle A_i x, A_i x' \rangle_{d_i}
=
\sum_{i=1}^N \rho_i x^\tT A_i^\tT A_i x'
=
x^\tT
\underbrace{
\sum_{i=1}^N \rho_i A_i^\tT A_i
}_{\eqqcolon M \text{ (diagonal)}
} 
x'
= \langle M^{\nicefrac{1}{2}} x, M^{\nicefrac{1}{2}} x' \rangle_{d_1}.
\]
Thus, we obtain 
\begin{align*}
\bigl(\R^{d_\oplus}, 
\sum\nolimits_{i=1}^N \rho_i \langle \cdot, \cdot \rangle_{d_i}, 
\hat{\pi}
\bigr)
&\simeq 
\bigl(\R^{d_1}, 
\sum\nolimits_{i=1}^N \rho_i 
\langle A_i \,\cdot, A_i \,\cdot \rangle_{d_i}, 
\xi_1
\bigr)
\\
&\simeq
\bigl(\R^{d_1}, 
\langle M^{\nicefrac{1}{2}} \cdot, M^{\nicefrac{1}{2}} \cdot \rangle_{d_1}, 
\xi_1
\bigr)
\\
&\simeq
\bigl(\R^{d_1}, 
\langle \cdot, \cdot \rangle_{d_1}, 
(M^{\nicefrac{1}{2}})_\#\xi_1
\bigr).
\end{align*}
Finally, $(M^{\nicefrac{1}{2}})_\# \xi_1$ 
is a centered Gaussian distribution on $\R^{d_1}$, 
whose covariance matrix is given by
\begin{align*}
    M^{\nicefrac{1}{2}} \Sigma_1 (M^{\nicefrac{1}{2}})^\tT
    = M D_1
    = \sum_{i=1}^N \rho_i A_i^\tT A_i D_1
    = 
    \sum_{i=1}^N 
    \rho_i
    \begin{pmatrix}
        D_i (D_1^{(d_i)})^{-1} & 0\\
        0 & 0
    \end{pmatrix}
    D_1
    =
    \sum_{i=1}^N \rho_i 
    \begin{pmatrix}
        D_i & 0\\
        0 & 0
    \end{pmatrix},
\end{align*}
which concludes the proof.

\bibliographystyle{abbrv}
\bibliography{reference}

\end{document}